%% file: paper.tex
\documentclass[microtype]{gtpart}     
\agtart
\usepackage{enumerate}      
\usepackage{graphicx}       
\usepackage{mathrsfs}       
\usepackage{pictexwd,dcpic} 
\usepackage{textcomp}       
\title[Converting Between Quadrilateral and Standard
    Solution Sets]{Converting Between Quadrilateral and Standard \\
    Solution Sets in Normal Surface Theory}

\author{Benjamin A Burton}
\givenname{Benjamin A}
\surname{Burton}
\address{School of Mathematics and Physics \\
    The University of Queensland \newline
    Brisbane QLD 4072 \\
    Australia}
\email{bab@debian.org}
\urladdr{}


\arxivreference{}
\arxivpassword{}

\volumenumber{}
\issuenumber{}
\publicationyear{}
\papernumber{}
\startpage{}
\endpage{}
\doi{}
\MR{}
\Zbl{}
\received{}
\revised{}
\accepted{}
\published{}
\publishedonline{}
\proposed{}
\seconded{}
\corresponding{}
\editor{}
\version{}

\newtheorem{theorem}{Theorem}[section]
\newtheorem{lemma}[theorem]{Lemma}
\newtheorem{corollary}[theorem]{Corollary}
\newtheorem{algorithm}[theorem]{Algorithm}
\newtheorem{conjecture}[theorem]{Conjecture}
\newtheorem{example}[theorem]{Example}
\newtheorem{notn}[theorem]{Notation}
\theoremstyle{definition}
\newtheorem{defn}[theorem]{Definition}
\newcommand{\adm}[1]{\alpha(#1)}
\newcommand{\ainv}[1]{\mathscr{A}_{#1}}
\newcommand{\binv}[1]{\mathscr{B}_{#1}}
\newcommand{\canp}[2]{\kappa_v^{(#1)}(#2)}
\newcommand{\claima}{($\star$)}
\newcommand{\claimb}{($\dagger$)}
\newcommand{\claimc}{({\fontfamily{cmr}\textreferencemark})}
\newcommand{\enref}[1]{\ref*{#1}}
\newcommand{\fxrays}{{\em FXrays}}
\newcommand{\pref}[1]{\ref*{#1}}
\newcommand{\proj}[1]{\overline{#1}}
\newcommand{\qmap}[1]{\qmapsymbol(#1)}
\newcommand{\qmapsymbol}{\pi}
\newcommand{\qproj}[1]{\proj{\mathbf{q}}(#1)}
\newcommand{\qrep}[1]{\mathbf{q}(#1)}
\newcommand{\qspacea}{\R_a^{3n}}
\newcommand{\qspaceza}{\Z_a^{3n}}
\newcommand{\quadproj}{\mathscr{Q}(\tri)}
\newcommand{\regina}{{\em Regina}}
\newcommand{\stdproj}{\mathscr{S}(\tri)}
\newcommand{\surfaces}{\mathbb{S}}
\newcommand{\surfacescan}{\mathbb{S}_c}
\newcommand{\tri}{\mathcal{T}}
\newcommand{\vmap}[1]{\vmapsymbol(#1)}
\newcommand{\vmapsymbol}{\varepsilon}
\newcommand{\vproj}[1]{\proj{\mathbf{v}}(#1)}
\newcommand{\vrep}[1]{\mathbf{v}(#1)}
\newcommand{\vspacea}{\R_a^{7n}}
\newcommand{\vspaceac}{\R_{a,c}^{7n}}
\newcommand{\vspaceza}{\Z_a^{7n}}
\newcommand{\vspacezac}{\Z_{a,c}^{7n}}
\newlength{\claimwidth}
\newlength{\claimindent}
\newcommand{\displayclaim}[2][\claima]
    {\par\medskip\par
     \setlength{\claimwidth}{\linewidth}
     \addtolength{\claimwidth}{-\claimindent}
     \addtolength{\claimwidth}{-\claimindent}
     \noindent\hspace{\claimindent}\parbox[b]{\claimwidth}{\textbf{Claim:}
        \textit{#2}}\hfill#1
     \par\medskip\par}
\newlength{\roadmapwidth}
\newlength{\roadmapindent}
\newcommand{\displayroadmap}[1]
    {\par\medskip\par
     \setlength{\roadmapwidth}{\linewidth}
     \addtolength{\roadmapwidth}{-\roadmapindent}
     \addtolength{\roadmapwidth}{-\roadmapindent}
     \noindent\hspace{\roadmapindent}\parbox[b]{\roadmapwidth}{#1}\\}
\newcommand{\proofpart}[1]
    {\par\bigskip\par
     \centerline{\textbf{---\:#1\:---}}
     \nopagebreak\par\smallskip\par}
\begin{document}

\begin{abstract}
    The enumeration of normal surfaces is a crucial but very slow
    operation in algorithmic $3$--manifold topology.  At the heart of this
    operation is a polytope vertex enumeration in a high-dimensional
    space (standard coordinates).  Tollefson's Q--theory speeds up this
    operation by using a much smaller space (quadrilateral coordinates),
    at the cost of a reduced solution set that might not always be
    sufficient for our needs.
    In this paper we present algorithms for converting between solution
    sets in quadrilateral and standard coordinates.
    As a consequence we obtain a new algorithm for enumerating all
    standard vertex normal surfaces, yielding both the
    speed of quadrilateral coordinates and the wider applicability of
    standard coordinates.  Experimentation
    with the software package {\em Regina} shows this new algorithm
    to be extremely fast in practice, improving speed for large cases by
    factors from thousands up to millions.
\end{abstract}

\maketitle


\input{intro.tex}

\input{normal.tex}

\input{canonical.tex}

\input{stdtoquad.tex}

\input{quadtostd.tex}

\input{expt.tex}

\input{notation.tex}

%
%
%
\bibliographystyle{gtart}
\bibliography{paper}

\end{document}

%% file: intro.tex
\section{Introduction} \label{s-intro}

The theory of normal surfaces plays a pivotal role in algorithmic
$3$--manifold topology.  Introduced by Kneser \cite{kneser29-normal} and
further developed by Haken \cite{haken61-knot,haken62-homeomorphism},
normal surfaces feature in key topological algorithms such as
unknot recognition \cite{haken61-knot}, $3$--sphere recognition
\cite{rubinstein95-3sphere,rubinstein97-3sphere,thompson94-thinposition},
connected sum and JSJ decomposition \cite{jaco95-algorithms-decomposition},
and testing for incompressible surfaces \cite{jaco84-haken}.

The beauty of normal surface theory is that it allows difficult
topological questions to be transformed into straightforward linear
programming problems, yielding algorithms that are well-suited for computer
implementation.  Unfortunately these linear programming problems
can be extremely expensive computationally, which is what motivates the
work described here.

Algorithms that employ normal surface theory typically operate as
follows:
\begin{enumerate}[(i)]
    \item Begin with a compact $3$--manifold triangulation formed from $n$
    tetrahedra;
    \item Enumerate all vertex normal surfaces within this
    triangulation, as described below; \label{en-intro-enumerate}
    \item Search through this list for a surface of particular interest
    (such as an essential sphere for the connected sum decomposition
    algorithm, or an essential disc for the unknot recognition algorithm).
    \label{en-intro-interesting}
\end{enumerate}

The linear programming problem (and often the bottleneck for the entire
algorithm) appears in step~(\enref{en-intro-enumerate}).  It can be shown
that the set of all normal surfaces within a triangulation is described by a
polytope in a $7n$--dimensional vector space;
step~(\enref{en-intro-enumerate}) then requires us to enumerate the vertices of
this polytope.  The normal surfaces described by these vertices are
called {\em vertex normal surfaces}.

The trouble with step~(\enref{en-intro-enumerate})
is that the vertex enumeration
algorithm can grow exponentially slow in $n$; moreover, this growth is
unavoidable since the number of vertex normal surfaces can likewise grow
exponentially large.  As a result, normal surface algorithms are
(at the present time) unusable for large triangulations.

Nevertheless, it is important to have these algorithms working as well
as possible in practice.
One significant advance in this regard
was made by Tollefson \cite{tollefson98-quadspace}, who showed that in certain
cases, normal surface enumeration could be done in a much smaller vector space
of dimension $3n$.  This $3n$--dimensional space is called
{\em quadrilateral coordinates}, and the resulting vertex normal
surfaces (referred to by Tollefson as {\em Q--vertex surfaces})
form the {\em quadrilateral solution set}.  For comparison, we refer to
the original $7n$--dimensional space as {\em standard coordinates}
and its vertex normal surfaces as the {\em standard solution set}.
It is important to note that these solution sets are different
(in fact we prove in Lemma~\ref{l-quadisstd} that one is essentially a proper
subset of the other).

Practically speaking, quadrilateral coordinates are a significant
improvement---although the running time
remains exponential, experiments show that the enumeration of
normal surfaces in quadrilateral coordinates runs
orders of magnitude faster than in standard coordinates.

However, using quadrilateral coordinates can be problematic from a
theoretical point of view.  In the algorithm overview given earlier,
step~(\enref{en-intro-interesting}) requires us to prove that, if an
interesting surface exists, then it exists as a vertex normal surface.
Such results are more difficult to prove in quadrilateral coordinates,
largely because addition becomes a more complicated operation; in particular,
useful properties of surfaces that are linear functionals
in standard coordinates (such as as Euler characteristic) are no longer
linear in quadrilateral coordinates.  As a result, only a few
results appear in the literature to show that quadrilateral coordinates
can replace standard coordinates in certain topological algorithms.

The purpose of this paper is, in essence, to show that we can have our
cake and eat it too.  That is, we show that we can enumerate vertex
normal surfaces in {\em standard} coordinates (thereby avoiding the
theoretical problems of quadrilateral coordinates) by first constructing the
{\em quadrilateral} solution set and then converting this into the
standard solution set (thus avoiding the performance problems of
standard coordinates).  The conversion process is not trivial
(and indeed forms the bulk of this paper),
but it is found to be extremely fast in practice.

The key results in this paper are as follows:
\begin{itemize}
    \item Algorithm~\ref{a-stdtoquad}, which gives a simple procedure
    for converting the standard solution set into the quadrilateral
    solution set;
    \item Algorithm~\ref{a-quadtostd}, which gives a more
    complex procedure for converting the quadrilateral solution set into
    the standard solution set;
    \item Algorithm~\ref{a-enumstd}, which builds on these results
    to provide a new way of enumerating all vertex normal surfaces in
    standard coordinates, by going via quadrilateral coordinates
    as outlined above.
\end{itemize}

The final algorithm in this list (Algorithm~\ref{a-enumstd})
is the ``end product'' of this paper---it
can be dropped into any high-level topological algorithm that requires the
enumeration of vertex normal surfaces.  Experimentation shows that this
new algorithm runs orders of magnitude faster than the current
state-of-the-art, with consistent improvements of the order of
$10^3$--$10^6$ times the speed observed for large cases.
Full details can be found in Section~\ref{s-expt}.

The remainder of this paper is structured as follows.
Section~\ref{s-normal} introduces the theory of normal
surfaces, and defines the standard and quadrilateral solution sets
precisely.  In Section~\ref{s-canonical} we address the ambiguity
inherent in quadrilateral coordinates by studying canonical surfaces
and vectors.
Sections~\ref{s-stdtoquad} and~\ref{s-quadtostd} contain the
main results, where we describe
the conversion from standard to quadrilateral coordinates and
quadrilateral to standard coordinates respectively.  We finish in
Section~\ref{s-expt} with experimental testing that shows how well these new
algorithms perform in practice.

Because this paper introduces a fair amount of notation, an appendix is
included that lists the key symbols and where they are defined.

For researchers who wish to perform their own experiments, the three
algorithms listed above have been implemented in version~4.6 of
the software package {\regina} \cite{regina,burton04-regina}.

Thanks must go to Ryan Budney and the University of Victoria, British
Columbia for their hospitality during the development of this work,
and to both RMIT University and the University of Melbourne
for continuing to support the
development of {\regina}.  The author also thanks the Victorian
Partnership for Advanced Computing for the use of their excellent
computing resources.

%% file: normal.tex
\section{Normal Surfaces} \label{s-normal}

In this section we provide the essential definitions of normal surface
theory, including both Haken's original formulation (standard coordinates)
and Tollefson's normal surface Q--theory (quadrilateral coordinates).

We only present what is required to define the standard and quadrilateral
solution sets.  For a more thorough overview of normal surface theory
the reader is referred to \cite{hass99-knotnp}; for further details on
quadrilateral coordinates the reader is referred to Tollefson's original
paper \cite{tollefson98-quadspace}.

\begin{defn}[Triangulation] \label{d-tri}
    A {\em compact $3$--manifold triangulation} is a finite collection of
    tetrahedra $\Delta_1,\ldots,\Delta_n$, where some or all of the
    $4n$ tetrahedron faces are affinely identified in pairs, and where
    the resulting topological space is a compact $3$--manifold.

    We allow different vertices of the same tetrahedron to be identified, and
    likewise with edges and faces (some authors refer to such structures
    as {\em pseudo-triangulations} or {\em semi-simplicial triangulations}).
    Any tetrahedron face that is {\em not} identified with some other
    tetrahedron face becomes part of the boundary of this $3$--manifold,
    and is referred to as a {\em boundary face}.

    Each equivalence class of tetrahedron vertices under these
    identifications is called a {\em vertex of the triangulation};
    likewise with edges and faces.
\end{defn}

It should be noted that, according to this definition, the link of each
vertex in the underlying $3$--manifold must be a disc or a $2$--sphere.
This rules out the ideal triangulations of Thurston \cite{thurston78-lectures};
we discuss the reasons for this decision at the end of this section.

\begin{defn}[Normal Surface]
    Let $\tri$ be a compact $3$--manifold triangulation, and let
    $\Delta$ be a tetrahedron of $\tri$.  A {\em normal disc} in $\Delta$
    is a properly embedded disc in $\Delta$ which does not touch any
    vertices of $\Delta$, and whose boundary consists of either
    (i)~three arcs running across three different faces of $\Delta$, or
    (ii)~four arcs running across all four faces of $\Delta$.
    We refer to such discs as {\em triangles} and {\em quadrilaterals}
    respectively.

    There are seven different {\em types} of normal disc in a
    tetrahedron, defined by the choice of which tetrahedron edges a disc
    intersects.  These include (i)~four triangle types, each surrounding
    a single vertex of $\Delta$, and (ii)~three quadrilateral
    types, each separating a single pair of opposite edges of $\Delta$.
    All seven disc types are illustrated in Figure~\ref{fig-normaltypes}.

    \begin{figure}[htb]
    \centerline{\includegraphics[scale=0.5]{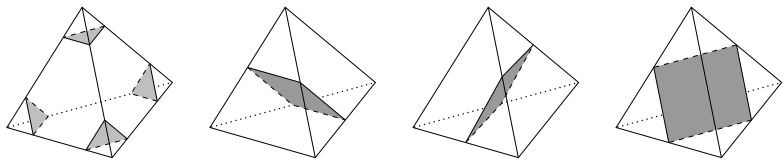}}
    \caption{The seven different types of normal disc in a tetrahedron}
    \label{fig-normaltypes}
    \end{figure}

    An {\em embedded normal surface} in the triangulation $\tri$
    is a properly embedded surface that intersects
    each tetrahedron of $\tri$ in a (possibly empty) collection of disjoint
    normal discs.  Here we allow both disconnected surfaces and the
    empty surface.

    We consider two normal surfaces {\em identical} if they are related by
    a {\em normal isotopy}, i.e., an ambient isotopy that
    preserves each simplex of $\tri$.
\end{defn}

We divert briefly to define a particular class of normal surface that plays
an important role in the relationship
between standard and quadrilateral coordinates.

\begin{defn}[Vertex Link] \label{d-link}
    Let $\tri$ be a compact $3$--manifold triangulation, and let
    $V$ be some vertex of $\tri$.  We define the {\em vertex link of $V$},
    denoted $\ell(V)$, to be the normal surface that appears at the
    frontier of a small regular neighbourhood of $V$.
    In particular, $\ell(V)$ contains one copy of each triangular disc type
    surrounding $V$, and contains no other normal discs at all.
\end{defn}

Here we follow the nomenclature of Jaco and Rubinstein
\cite{jaco03-0-efficiency};
in particular, Definition~\ref{d-link} is not the same as the
combinatorial link in a simplicial complex.  Tollefson refers to
vertex links as {\em trivial surfaces} \cite{tollefson98-quadspace}.

Note that Definition~\ref{d-tri} implies that $\ell(V)$ is a disc or a
$2$--sphere (according to whether or not $V$ is on the boundary of the
$3$--manifold).
In the case where $\tri$ is a one-vertex triangulation, the normal surface
$\ell(V)$ contains precisely one copy of every triangular disc type
in the triangulation, and no other normal discs.

At this point the theory of normal surfaces moves into linear algebra,
whereupon we must choose between the formulation of Haken (standard
coordinates) or Tollefson (quadrilateral coordinates).  In the text
that follows we outline both formulations side by side.

\begin{defn}[Vector Representations] \label{d-vecrep}
    Let $\tri$ be a compact $3$--manifold triangulation built from the $n$
    tetrahedra $\Delta_1,\ldots,\Delta_n$,
    and let $S$ be an embedded normal surface in $\tri$.

    Consider the individual normal discs that form the surface $S$.
    Let $t_{i,j}$ denote the number of triangular discs of the $j$th
    type in $\Delta_i$ ($j=1,2,3,4$), and let $q_{i,k}$ denote the number of
    quadrilateral discs of the $k$th type in $\Delta_i$ ($k=1,2,3$).

    Then the {\em standard vector representation} of $S$,
    denoted $\vrep{S}$, is the $7n$--dimensional vector
    \begin{alignat*}{2}
    \vrep{S}~=\,(~
        & t_{1,1},t_{1,2},t_{1,3},t_{1,4},\ q_{1,1},q_{1,2}&&,q_{1,3}\ ;\\
        & t_{2,1},t_{2,2},t_{2,3},t_{2,4},\ q_{2,1},q_{2,2}&&,q_{2,3}\ ;\\
        & \ldots &&,q_{n,3}\ ),
    \end{alignat*}
    and the {\em quadrilateral vector representation} of $S$, denoted
    $\qrep{S}$, is the $3n$--dimensional vector
    \[ \qrep{S}~=\,(~
        q_{1,1},q_{1,2},q_{1,3}\ ;\ q_{2,1},q_{2,2},q_{2,3}\ ;
        \ \ldots,q_{n,3}\ ). \]

    When we are working with standard vector representations in
    $\R^{7n}$ we say we are working in {\em standard coordinates}.
    Likewise, when working with quadrilateral vector representations in
    $\R^{3n}$ we say we are working in {\em quadrilateral coordinates}.
\end{defn}

It turns out that, if we ignore vertex links, then the vector
representations contain enough information to completely reconstruct a
normal surface.  The results, due to Haken \cite{haken61-knot}
and Tollefson \cite{tollefson98-quadspace}, are as follows.

\begin{lemma} \label{l-vecrep}
    Consider two embedded normal surfaces $S$ and $T$ within
    some compact $3$--ma\-ni\-fold triangulation.
    \begin{itemize}
        \item The standard vector representations of $S$ and $T$
        are equal, that is, $\vrep{S}=\vrep{T}$, if and only if
        surfaces $S$ and $T$ are identical.
        \item The quadrilateral vector representations of $S$ and $T$
        are equal, that is, $\qrep{S}=\qrep{T}$, if and only if
        (i) $S$ and $T$ are identical, or (ii) $S$ and $T$
        differ only by adding or removing vertex linking components.
    \end{itemize}
\end{lemma}

Although every embedded normal surface has a standard and quadrilateral
vector representation, there are many vectors in $\R^{7n}$ and
$\R^{3n}$ respectively that do not represent any normal surface at all.
Haken \cite{haken61-knot} and Tollefson \cite{tollefson98-quadspace} completely
characterise which vectors represent embedded normal surfaces,
using the concept of {\em admissible vectors}.
We build up a definition of this concept now, and then present the full
characterisation results of Haken and Tollefson in Theorem~\ref{t-admissible}.

\begin{defn}[Standard Matching Equations] \label{d-matching-std}
    Let $\tri$ be a compact $3$--manifold triangulation built from the $n$
    tetrahedra $\Delta_1,\ldots,\Delta_n$, and consider some
    $7n$--dimensional vector
    $\mathbf{v}=
    \left(t_{1,1},t_{1,2},t_{1,3},t_{1,4},q_{1,1},q_{1,2},q_{1,3};
    \ldots,q_{n,3}\right)$.
    For each non-boundary face $F$ of $\tri$ and each edge $e$ of the face
    $F$, we obtain an equation as follows.

    In essence, our equation states that we must be able to match the
    normal discs on one side of $F$ with the normal discs on the other.
    To express this formally, let $\Delta_i$ and $\Delta_j$ be the two
    tetrahedra joined along face $F$.  In each tetrahedron $\Delta_i$
    and $\Delta_j$ there is precisely
    one triangle type and one quadrilateral type that meets face $F$ in
    an arc parallel to $e$; let these be described by the coordinates
    $t_{i,a}$ and $q_{i,b}$ in $\Delta_i$ and $t_{j,c}$ and $q_{j,d}$
    in $\Delta_j$.  Our equation is then
    \[ t_{i,a} + q_{i,b} = t_{j,c} + q_{j,d} . \]
    The set of all such equations is called the set of
    {\em standard matching equations} for $\tri$.
\end{defn}

\begin{figure}[htb]
\centerline{\includegraphics[scale=1]{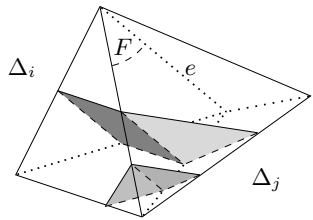}}
\caption{An example of the standard matching equations}
\label{fig-matchingstd}
\end{figure}

Note that if $\tri$ has $f$ non-boundary faces then there are
$3f$ such equations in total; in particular, if $\tri$ has no boundary
at all then there are $6n$ standard matching equations.
Figure~\ref{fig-matchingstd} shows an
illustration of one such equation; here we have one
triangle and one quadrilateral in $\Delta_i$ meeting
two triangles in $\Delta_j$, giving
$(t_{i,a}+q_{i,b}=1+1) = (t_{j,c}+q_{j,d}=2+0)$.

\begin{defn}[Quadrilateral Matching Equations] \label{d-matching-quad}
    Let $\tri$ be a compact $3$--manifold triangulation built from the $n$
    tetrahedra $\Delta_1,\ldots,\Delta_n$, and consider some
    $3n$--di\-men\-sion\-al vector
    $\mathbf{q}=\left(q_{1,1},q_{1,2},q_{1,3};\ldots,q_{n,3}\right)$.
    For each non-boundary edge $e$ of $\tri$, we obtain an equation as
    follows.

    Consider the tetrahedra containing edge $e$; these are
    arranged in a cycle around $e$ as illustrated in
    Figure~\ref{fig-matchingquad}.  Choose an arbitrary direction
    around this cycle, and arbitrarily label the two ends of $e$ as
    {\em upper} and {\em lower}.

    \begin{figure}[htb]
    \centerline{\includegraphics[scale=0.8]{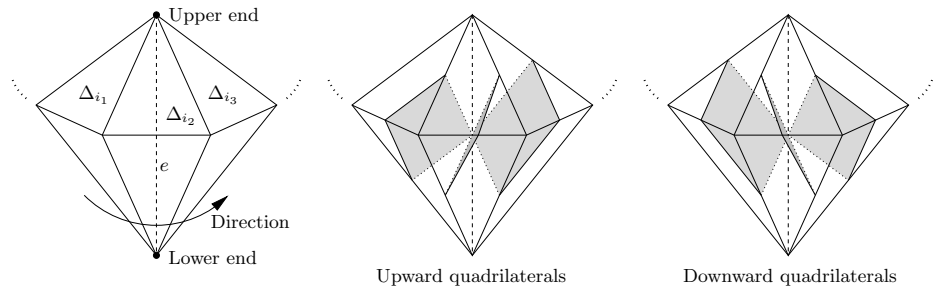}}
    \caption{Building the quadrilateral matching equations}
    \label{fig-matchingquad}
    \end{figure}

    Within each of these tetrahedra, there are two quadrilateral types
    that meet edge $e$:  the {\em upward} quadrilaterals, which rise
    from the lower end of $e$ to the upper end as we move around the
    cycle, and the {\em downward} quadrilaterals, which fall in the
    opposite direction.  These are again illustrated in
    Figure~\ref{fig-matchingquad}.

    We can now create an equation from edge $e$ as follows.
    Let the tetrahedra containing $e$ be
    $\Delta_{i_1},\Delta_{i_2},\ldots,\Delta_{i_k}$, let the
    coordinates corresponding to the upward quadrilateral types be
    $q_{i_1,u_1},q_{i_2,u_2},\ldots,q_{i_k,u_k}$, and let the coordinates
    corresponding to the downward quadrilateral types be
    $q_{i_1,d_1},q_{i_2,d_2},\ldots,q_{i_k,d_k}$.  Then we obtain the equation
    \[ q_{i_1,u_1} + q_{i_2,u_2} + \ldots + q_{i_k,u_k} =
       q_{i_1,d_1} + q_{i_2,d_2} + \ldots + q_{i_k,d_k} .\]

    The set of all such equations is called the set of
    {\em quadrilateral matching equations} for $\tri$.
\end{defn}

We will see that both the standard and quadrilateral matching equations
form necessary but not sufficient
conditions for a non-negative integer vector to represent
an embedded normal surface.  We still
need one more set of constraints, which we define as follows.

\begin{defn}[Quadrilateral Constraints]
    Let $\tri$ be a compact $3$--manifold triangulation built from the $n$
    tetrahedra $\Delta_1,\ldots,\Delta_n$, and let $\mathbf{w}$ be
    either a $7n$--dimensional vector of the form
    $\left(t_{1,1},t_{1,2},t_{1,3},t_{1,4},q_{1,1},q_{1,2},q_{1,3};
    \ldots,q_{n,3}\right)$,
    or a $3n$--dimensional vector of the form
    $\left(q_{1,1},q_{1,2},q_{1,3};\ldots,q_{n,3}\right)$.

    Then $\mathbf{w}$ satisfies the {\em quadrilateral constraints} if,
    for each tetrahedron $\Delta_i$, at most of one of the quadrilateral
    coordinates $q_{i,1}$, $q_{i,2}$ and $q_{i,3}$ is non-zero.
\end{defn}

The quadrilateral constraints arise because any two quadrilaterals of
different types within the same tetrahedron must intersect, yet
embedded normal surfaces cannot have self-in\-ter\-sec\-tions.
We have now gathered enough conditions for the complete
characterisation results of Haken \cite{haken61-knot} and
Tollefson \cite{tollefson98-quadspace}, which we reproduce in
Definition~\ref{d-admissible} and Theorem~\ref{t-admissible}.

\begin{defn}[Admissible Vector] \label{d-admissible}
    Let $\tri$ be a compact $3$--manifold triangulation built from $n$ tetrahedra.
    A ($7n$ or $3n$)--dimensional vector is called {\em admissible}
    if (i)~its entries are all non-negative, (ii)~it satisfies
    the (standard or quadrilateral) matching equations for $\tri$, and
    (iii)~it satisfies the quadrilateral constraints for $\tri$.
\end{defn}

\begin{theorem} \label{t-admissible}
    Let $\tri$ be a compact $3$--manifold triangulation built from $n$ tetrahedra,
    and let $\mathbf{w}$ be a ($7n$ or $3n$)--dimensional
    vector of integers.  Then $\mathbf{w}$ is the
    (standard or quadrilateral) vector representation of an
    embedded normal surface in $\tri$ if and only if $\mathbf{w}$ is
    admissible.
\end{theorem}

Although we can now reduce normal surfaces to vectors in
$\R^{7n}$ or $\R^{3n}$, we still have infinitely many surfaces to
search through if we are seeking an ``interesting'' surface, such as
an essential $2$--sphere or an incompressible surface.  The following
series of definitions, due to Jaco and Oertel \cite{jaco84-haken},
allow us to reduce such searches to finite problems
by restricting our attention to what are known as
{\em vertex normal surfaces}.

\begin{defn}[Projective Solution Space] \label{d-soln-space}
    For any dimension $d$, we define the following regions in $\R^d$:
    \begin{itemize}
        \item The {\em non-negative orthant} $O^d$ is the region in
        $\R^d$ in which all coordinates are non-negative; that is,
        $O^d = \{\mathbf{x} \in \R^d\,|\,x_i \geq 0\ \forall i\}$.
        \item The {\em projective hyperplane} $J^d$ is the hyperplane in
        $\R^d$ where all coordinates sum to 1; that is,
        $J^d = \{\mathbf{x} \in \R^d\,|\,\sum x_i = 1\}$.
    \end{itemize}
    Note that the intersection $O^d \cap J^d$ is the unit simplex in $\R^d$.

    Let $\tri$ be a compact $3$--manifold triangulation built from 
    $n$ tetrahedra.
    The {\em standard projective solution space} for $\tri$, denoted
    $\stdproj$, is the
    region in $\R^{7n}$ consisting of all points in $O^{7n} \cap J^{7n}$
    that satisfy the standard matching equations.
    Likewise, the {\em quadrilateral projective solution space} for $\tri$,
    denoted $\quadproj$, is
    the region in $\R^{3n}$ consisting of all points in $O^{3n} \cap J^{3n}$
    that satisfy the quadrilateral matching equations.

    Since each $O^d \cap J^d$ is the unit simplex and the matching equations
    are both linear and rational, it follows that the standard and quadrilateral
    projective solution spaces are (finite) convex rational polytopes in
    $\R^{7n}$ and $\R^{3n}$ respectively.
\end{defn}

It is clear from Theorem~\ref{t-admissible} that the non-zero vectors in
$\R^{7n}$ or $\R^{3n}$ that represent embedded normal surfaces
are precisely those positive multiples of points in $\stdproj$ or $\quadproj$
that (i) are integer vectors, and (ii) satisfy the quadrilateral constraints.

\begin{defn}[Projective Image] \label{d-projimage}
    Suppose that $\mathbf{x} \in \R^d$ is not the zero vector.
    We define the {\em projective image} of $\mathbf{x}$, denoted
    $\proj{\mathbf{x}}$, to be the vector $\mathbf{x} / \sum x_i$.
    In other words, $\proj{\mathbf{x}}$ is the (unique) multiple of
    $\mathbf{x}$ that lies in the projective hyperplane $J^d$.

    To avoid complications with vertex links and the empty surface,
    we define the projective image of the zero vector to be the zero vector.
    That is, $\proj{\mathbf{0}} = \mathbf{0}$ (which does {\em not} lie in the
    projective hyperplane $J^d$).

    Let $S$ be an embedded normal surface in some triangulation $\tri$.
    To keep our notation clean, we write the projective images of the
    vector representations $\vrep{S}$ and $\qrep{S}$ as
    $\vproj{S}$ and $\qproj{S}$ respectively.
\end{defn}

\begin{defn}[Vertex Normal Surface] \label{d-vertex}
    Let $\tri$ be a compact $3$--manifold triangulation built from $n$
    tetrahedra, and let $S$ be an embedded normal surface in $\tri$.
    We call $S$ a {\em standard vertex normal surface} if and only if
    $\vproj{S}$ (the projective image of the standard vector
    representation of $S$) is a vertex of the polytope $\stdproj$.
    Likewise, we call $S$ a
    {\em quadrilateral vertex normal surface} if and only if
    $\qproj{S}$ is a vertex of the polytope $\quadproj$.
\end{defn}

Although vertex normal surfaces correspond to vertices of the projective
solution space, this correspondence does not always work in the other
direction.  Instead we must restrict our attention to vectors that satisfy the
quadrilateral constraints.

\begin{defn}[Solution Sets] \label{d-solution}
    Let $\tri$ be a compact $3$--manifold triangulation built from $n$
    tetrahedra.
    The {\em standard solution set} for $\tri$ is the (finite) set of all
    vertices of the polytope $\stdproj$ that satisfy the quadrilateral
    constraints.
    Likewise, the {\em quadrilateral solution set} for $\tri$ is the
    (finite) set of all vertices of the polytope $\quadproj$ that
    satisfy the quadrilateral constraints.
\end{defn}

The correspondence between solution sets and vertex normal surfaces
is now an immediate consequence of
Theorem~\ref{t-admissible} and the fact that each projective solution
space is a rational polytope:

\begin{corollary} \label{c-solnisvertex}
    Let $\tri$ be a compact $3$--manifold triangulation built from $n$
    tetrahedra, and let $\mathbf{w}$ be a ($7n$ or $3n$)--dimensional
    vector.  Then $\mathbf{w}$ is the projective image of the
    vector representation for a
    (standard or quadrilateral) vertex normal surface if and only if
    $\mathbf{w}$ is in the (standard or quadrilateral) solution set.
\end{corollary}

We return now to the overview of a ``typical normal surface algorithm''
as given in Section~\ref{s-intro}.  Such algorithms typically work
because we can prove that, if a $3$--manifold triangulation contains an
``interesting'' surface, then it contains an interesting
{\em vertex normal surface}.  Examples of such theorems include:
\begin{itemize}
    \item Jaco and Oertel \cite{jaco84-haken} show that, if a closed
    irreducible $3$--manifold triangulation contains a two-sided
    incompressible surface, then such a surface exists as a
    standard vertex normal surface.
    Jaco and Tollefson \cite{jaco95-algorithms-decomposition} extend this
    result to
    bounded manifolds, and Tollefson \cite{tollefson98-quadspace} shows that
    such a surface must also exist as a quadrilateral vertex normal surface.

    \item Jaco and Tollefson \cite{jaco95-algorithms-decomposition} prove
    similar results for essential spheres in closed $3$--manifolds and
    essential compression discs in bounded irreducible $3$--manifolds;
    in particular, they show that if such a surface exists then one can
    be found amongst the standard vertex normal surfaces.
    With these results, they build algorithms to solve problems
    such as connected sum decomposition, JSJ decomposition and
    unknot recognition.
\end{itemize}

We can therefore build such an algorithm by constructing the standard
or quadrilateral solution set for our triangulation, and then searching
through the solutions for one that scales to an ``interesting'' normal
surface.

The construction of the solution sets is, though finite, an
exponentially slow procedure in the number of tetrahedra $n$.
The best known algorithm to date is described in \cite{burton08-dd};
it is essentially a variant of the double description method of
Motzkin et al.~\cite{motzkin53-dd}, modified in several ways to exploit
the quadrilateral constraints for greater speed and lower memory
consumption.

The remainder of this paper is concerned mainly with the {\em conversion}
between the standard solution set and the quadrilateral solution set.
Upon establishing conversion algorithms in both directions
(Algorithms~\ref{a-stdtoquad} and~\ref{a-quadtostd}), we finish with a new
algorithm for {\em constructing} the standard solution set
(Algorithm~\ref{a-enumstd}) that is orders of magnitude faster than the
current state-of-the-art.

We conclude this section with a brief discussion of {\em ideal}
triangulations.  These triangulations, due to Thurston
\cite{thurston78-lectures},
include vertices whose links are neither $2$--spheres nor discs, but
rather closed surfaces with genus (such as tori or Klein bottles).
By removing these vertices (and only these vertices), we obtain a
triangulation of a non-compact $3$--manifold.  One of the most well-known
ideal triangulations is the two-tetrahedron triangulation of the figure
eight knot complement, discussed in detail in \cite{matsumoto00-fig8}.

Quadrilateral coordinates play a special role in ideal
triangulations---they allow us to describe {\em spun normal surfaces},
which contain infinitely many triangular discs spiralling in towards the
high-genus vertices.  Such surfaces cannot be represented in standard
coordinates at all, which is why we must restrict our attention in
this paper to compact $3$--manifold triangulations.  The reader is referred
to Tillmann \cite{tillmann08-finite} for a thorough overview of spun normal
surfaces.

%% file: canonical.tex
\section{Canonical Surfaces and Vectors} \label{s-canonical}

Although our eventual goal is to construct algorithms for converting
between the standard and quadrilateral {\em solution sets}, we begin in
this section with the more modest aim of converting between standard and
quadrilateral {\em vectors}.

One complication we face is that, whereas vectors in standard
coordinates represent unique normal surfaces, vectors in quadrilateral
coordinates do not (Lemma~\ref{l-vecrep}).  We work around this
uniqueness problem by introducing the notion of {\em canonical
surfaces} and {\em canonical vectors} in standard coordinates.  Although
this allows us to map vectors in quadrilateral coordinates to unique
{\em canonical} vectors in standard coordinates and
unique {\em canonical} surfaces,
we will find that these maps are not as well-behaved as we
might like them to be.

The structure of this section is as follows.  We first define canonical
surfaces and canonical vectors and examine some of their basic
properties.  Following this we study several additional maps between
both surfaces and vectors; amongst these maps are the
{\em quadrilateral projection} $\qmapsymbol\co\R^{7n}\to\R^{3n}$ and the
{\em canonical extension} $\vmapsymbol\co\R^{3n}\to\R^{7n}$, which
convert back and forth between vectors
in standard and quadrilateral coordinates.  We finish the section with
Algorithm~\ref{a-quaddfs}, which shows how these conversions can be
performed in as fast a time complexity as possible.

Throughout this section, we assume that we are working with a compact
$3$--manifold triangulation $\tri$ built from $n$ tetrahedra.  We also
allow a little flexibility with our notation:  the expression $\ell(V)$
will be used to refer to both the vertex linking surface surrounding $V$ (as
presented in Definition~\ref{d-link}) and also its standard vector
representation in $\R^{7n}$.

\begin{defn}[Canonical Normal Surface] \label{d-canonical-surface}
    A {\em canonical normal surface} in the triangulation $\tri$ is
    an embedded normal surface that does not contain any vertex linking
    components.
\end{defn}

The purpose of this definition is to resolve the ambiguities
inherent in quadrilateral coordinates.  In particular, it gives us
the following uniqueness properties, which
follow immediately from Lemma~\ref{l-vecrep} and Theorem~\ref{t-admissible}:

\begin{lemma} \label{l-canonical-vecrep}
    Let $S$ and $T$ be canonical normal surfaces within the
    triangulation $\tri$.  Then the quadrilateral vector representations of
    $S$ and $T$ are equal, that is, $\qrep{S}=\qrep{T}$, if and only if
    surfaces $S$ and $T$ are identical.
\end{lemma}

\begin{lemma} \label{l-canonical-admissible}
    Let $\mathbf{w}$ be a $3n$--dimensional
    vector of integers.  Then $\mathbf{w}$ is the
    quadrilateral vector representation of a canonical
    normal surface in $\tri$ if and only if $\mathbf{w}$ is
    admissible.  Moreover, this canonical normal surface is unique.
\end{lemma}

Instead of thinking of canonical surfaces as having no vertex
links, we can instead think of them as surfaces where it is impossible to
{\em remove} a vertex link.  With this in mind, we extend the concept from
surfaces to vectors as follows.

\begin{defn}[Canonical Vector] \label{d-canonical-vector}
    Let $\mathbf{w}$ be any vector in $\R^{7n}$ (i.e., in
    standard coordinates).  We call $\mathbf{w}$ a
    {\em canonical vector} if and only if
    (i)~all triangular coordinates of $\mathbf{w}$ are non-negative, but
    (ii)~if we subtract $\epsilon \ell(V)$ for any $\epsilon > 0$
    and any vertex link $\ell(V)$ then some triangular coordinate
    of $\mathbf{w}$ must become negative.

    In other words, for each vertex $V$ of the triangulation $\tri$,
    the following property must hold.
    Let $t_{i_1,j_1},t_{i_2,j_2},\ldots,t_{i_k,j_k}$ be the
    coordinates in $\mathbf{w}$ corresponding to the
    triangular normal discs surrounding $V$.  Then all of
    $t_{i_1,j_1},t_{i_2,j_2},\ldots,t_{i_k,j_k}$ are at least zero, and at
    least one of these coordinates is {\em equal} to zero.
\end{defn}

Essentially this definition states that (i)~$\mathbf{w}$ {\em might} be
admissible (having non-negative triangular coordinates), but
(ii)~$\mathbf{w} - \epsilon \ell(V)$ can {\em never} be admissible.

We have already established two bijections between surfaces and vectors:
Theorem~\ref{t-admissible} shows a bijection between embedded normal surfaces
and admissible integer vectors in $\R^{7n}$, and
Lemma~\ref{l-canonical-admissible} shows
a bijection between canonical normal surfaces and admissible integer
vectors in $\R^{3n}$.  We can now extend this list with a bijection
between canonical normal surfaces and admissible {\em canonical}
integer vectors in $\R^{7n}$.

\begin{lemma} \label{l-canonical-to-canonical}
    The standard vector representation of a canonical normal surface
    is a canonical vector in $\R^{7n}$.
    Conversely, every admissible canonical integer vector in $\R^{7n}$
    is the standard vector representation of a (unique) canonical normal
    surface.
\end{lemma}

\begin{proof}
    This result follows immediately from Theorem~\ref{t-admissible} by
    observing that, if an admissible {\em integer}
    vector $\mathbf{w} \in \R^{7n}$
    is not canonical, then all of the triangular coordinates surrounding
    some vertex $V$ are $\geq 1$, and so
    $\mathbf{w} = \ell(V) + \mathbf{w'}$ for some other admissible
    integer vector $\mathbf{w'}$.
\end{proof}

We can observe that, if we restrict our attention to admissible integer
vectors, then we have bijections between
(i)~canonical vectors in standard coordinates and canonical surfaces, and
(ii)~vectors in quadrilateral coordinates and canonical surfaces.
It follows then that we must have a bijection between canonical vectors
in standard coordinates and vectors in quadrilateral coordinates; that
is, {\em a method for converting between coordinate systems}.
We develop this idea further in Definition~\ref{d-proj-ext}.

Although the ``canonical'' property gives us uniqueness results and
bijections that we did not have before, it is not particularly well-behaved.
In particular, it is clear from Definition~\ref{d-canonical-vector} that
this property is preserved under scalar multiplication
but not necessarily under addition.  However, we can salvage the
situation a little as seen in the following result.

\begin{lemma} \label{l-canonical-linear}
    If $\mathbf{w} \in \R^{7n}$ is a canonical vector then so is
    $\lambda \mathbf{w}$ for any $\lambda > 0$.  Likewise, if
    $\mathbf{w} \in \R^{7n}$ is an {\em admissible} canonical vector
    then so is $\lambda \mathbf{w}$ for any $\lambda > 0$.
    Finally, if $\mathbf{w}=\mathbf{x}+\mathbf{y}$ for
    admissible vectors $\mathbf{w},\mathbf{x},\mathbf{y} \in \R^{7n}$
    and $\mathbf{w}$ is canonical then so are $\mathbf{x}$ and $\mathbf{y}$.
\end{lemma}

\begin{proof}
    This follows immediately from Definition~\ref{d-canonical-vector}
    and the fact that the matching equations are invariant under scalar
    multiplication.
\end{proof}

We proceed now to define several mappings that express the
relationships between canonical surfaces, non-canonical surfaces,
vectors in standard coordinates and vectors in quadrilateral coordinates.
Lemma~\ref{l-mappings} summarises the interplay between these relationships.
We begin by presenting notation for the domains and ranges of these
functions.

\begin{notn} \label{notn-spaces}
    Let $\surfaces$ denote the set of all embedded normal surfaces
    (up to normal isotopy), and let $\surfacescan \subset \surfaces$
    denote the set of all canonical normal surfaces.
    Let $\vspacea$ and $\qspacea$ denote the set of all admissible
    vectors in $7n$ and $3n$ dimensions respectively, and let
    $\vspaceac \subset \vspacea$ denote the set of all admissible
    canonical vectors in $7n$ dimensions.  Likewise,
    let $\vspaceza$ and $\qspaceza$ denote the set of all admissible
    integer vectors in $7n$ and $3n$ dimensions respectively, and let
    $\vspacezac \subset \vspaceza$ denote the set of all admissible
    canonical integer vectors in $7n$ dimensions.
\end{notn}

It follows then that standard vector representation is
a bijection $\mathbf{v}\co\surfaces\to\vspaceza$ that takes the subset
$\surfacescan \subset \surfaces$ to the subset $\vspacezac \subset
\vspaceza$.  Likewise, quadrilateral vector representation
is a many-to-one function $\mathbf{q}\co\surfaces\to\qspaceza$ that becomes a
bijection when restricted to $\surfacescan$.

\begin{defn}[Represented Surface] \label{d-represented}
    Let $\mathbf{w}$ be an admissible integer vector in
    $\R^{7n}$.  Then the {\em represented surface} of
    $\mathbf{w}$, denoted $\sigma_v(\mathbf{w})$, is the
    unique embedded normal surface with standard vector representation
    $\mathbf{w}$ (as noted in Theorem~\ref{t-admissible}).
    Thus $\sigma_v\co\vspaceza\to\surfaces$ is the inverse function
    to $\mathbf{v}\co\surfaces\to\vspaceza$.

    Likewise, let $\mathbf{w}$ be an admissible integer vector in
    $\R^{3n}$.  Then the {\em represented surface} of
    $\mathbf{w}$, denoted $\sigma_q(\mathbf{w})$, is the
    unique canonical normal surface with quadrilateral vector representation
    $\mathbf{w}$ (as noted in Lemma~\ref{l-canonical-admissible}).
    Thus $\sigma_q\co\qspaceza\to\surfacescan$ is the inverse function
    to the restriction $\mathbf{q}\co\surfacescan\to\qspaceza$.
\end{defn}

\begin{defn}[Canonical Part] \label{d-canonical-part}
    Let $S$ be an embedded normal surface within the triangulation $\tri$.
    The {\em canonical part} of $S$, denoted $\kappa_s(S)$, is the
    canonical normal surface obtained by removing all vertex linking
    components from $S$.
    It follows that $\kappa_s$ is a function
    $\kappa_s\co\surfaces\to\surfacescan$ whose restriction to
    $\surfacescan$ is the identity.

    Similarly, let $\mathbf{w}$ be any vector in $\R^{7n}$.
    The {\em canonical part} of $\mathbf{w}$, denoted $\kappa_v(\mathbf{w})$,
    is the unique canonical vector that can be obtained from $\mathbf{w}$
    by adding and/or subtracting scalar multiples of vertex links.
    It follows that, if we restrict
    our attention to admissible vectors, then
    $\kappa_v$ is a function $\kappa_v\co\vspacea\to\vspaceac$
    whose restriction to $\vspaceac$ is the identity.
\end{defn}

The canonical part of a vector $\mathbf{w}\in\R^{7n}$ can be constructed
as follows.  Let the vertices of the triangulation be $V_1,\ldots,V_m$,
and for each $i$ let $\lambda_i$ be the minimum of all triangular
coordinates in $\mathbf{w}$ that correspond to triangular normal discs
surrounding $V_i$ (so $\mathbf{w}$ is canonical if and only if
every $\lambda_i=0$).  Then $\kappa_v(\mathbf{w})=
\mathbf{w}-\lambda_1\ell(V_1)-\ldots-\lambda_m\ell(V_m)$.

We now come to the point of defining conversion functions between
vectors in standard coordinates and vectors in quadrilateral coordinates.

\begin{defn}[Projection and Extension] \label{d-proj-ext}
    Let $\mathbf{w} \in \R^{7n}$ be any vector in standard coordinates;
    recall that the $7n$ coordinates of $\mathbf{w}$ correspond to
    $3n$ quadrilateral disc types and $4n$ triangular disc types.
    The {\em quadrilateral projection} of $\mathbf{w}$, denoted
    $\qmap{\mathbf{w}}$, is defined to be the vector in $\R^{3n}$ consisting
    of only the $3n$ quadrilateral coordinates for $\mathbf{w}$.  That is, if
    \begin{alignat*}{2}
    \mathbf{w}~=\,(~
        & t_{1,1},t_{1,2},t_{1,3},t_{1,4},\ q_{1,1},q_{1,2}&&,q_{1,3}\ ;\\
        & t_{2,1},t_{2,2},t_{2,3},t_{2,4},\ q_{2,1},q_{2,2}&&,q_{2,3}\ ;\\
        & \ldots &&,q_{n,3}\ ) \in \R^{7n},
    \end{alignat*}
    then
    \[ \qmap{\mathbf{w}}~=\,(~
        q_{1,1},q_{1,2},q_{1,3}\ ;\ q_{2,1},q_{2,2},q_{2,3}\ ;
        \ \ldots,q_{n,3}\ ) \in \R^{3n}.\]

    Conversely, let $\mathbf{w} \in \qspacea$ be any admissible vector in
    quadrilateral coordinates.  The {\em canonical extension} of
    $\mathbf{w}$, denoted $\vmap{\mathbf{w}}$, is defined to be the unique
    admissible canonical vector in $\vspaceac$ whose
    quadrilateral projection is $\mathbf{w}$.

    It follows that, if we restrict our attention to admissible canonical
    vectors, then the quadrilateral projection
    $\qmapsymbol\co\vspaceac\to\qspacea$ is the inverse function to the
    canonical extension $\vmapsymbol\co\qspacea\to\vspaceac$.
\end{defn}

It does need to be shown that canonical extension is well-defined; that is,
that for any admissible $\mathbf{w} \in \qspacea$ there is a unique
admissible canonical $\mathbf{x} \in \vspaceac$ for which
$\qmap{\mathbf{x}}=\mathbf{w}$.
Lemmata~\ref{l-canonical-admissible} and~\ref{l-canonical-to-canonical}
together show this to be true in the integers;
since admissibility and canonicity are invariant under
positive scalar multiplication this is also true in the rationals,
and because the matching equations are rational and linear this
fact extends to the reals.

Quadrilateral projection and canonical extension
are true ``conversion functions'', in the sense that if
$S$ is any embedded normal surface then $\qmapsymbol$ maps $\vrep{S} \mapsto
\qrep{S}$, and if $S$ is also canonical then
$\vmapsymbol$ maps $\qrep{S} \mapsto \vrep{S}$.  The advantage of the
broader definition above is that $\qmapsymbol$ and $\vmapsymbol$ can
also be applied to rational and real vectors, which means that
we can use them to convert not just vector representations of surfaces but
also arbitrary admissible points within the projective solution spaces.

This brings us to the end of our list of mappings.
To conclude this section, we bring these mappings together
and show how they interact (Lemma~\ref{l-mappings}), and then
we describe how the conversions $\qmapsymbol$ and $\vmapsymbol$ can be
performed in as fast a time complexity as possible (Algorithm~\ref{a-quaddfs}).

\begin{figure}[htb]
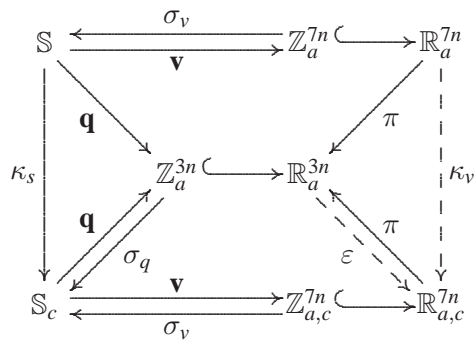

\centerline{\begindc{\commdiag}[1]
\obj(0,100)[S]{$\surfaces$}
\obj(0,0)[Sc]{$\surfacescan$}
\obj(100,100)[A]{$\vspaceza$}
\obj(100,0)[Ac]{$\vspacezac$}
\obj(50,50)[Aq]{$\qspaceza$}
\obj(150,100)[V]{$\vspacea$}
\obj(150,0)[Vc]{$\vspaceac$}
\obj(100,50)[Q]{$\qspacea$}
\mor(100,0)(150,0){}[\atleft,\injectionarrow]
\mor(100,100)(150,100){}[\atright,\injectionarrow]
\mor(50,50)(100,50){}[\atleft,\injectionarrow]
\mor(0,100)(0,0){$\kappa_s$}[\atright,\solidarrow]
\mor(150,100)(150,0){$\kappa_v$}[\atleft,\dasharrow]
\mor(0,3)(100,3){$\mathbf{v}$}[\atleft,\solidarrow]
\mor(0,97)(100,97){$\mathbf{v}$}[\atright,\solidarrow]
\mor(100,-3)(0,-3){$\sigma_v$}[\atleft,\solidarrow]
\mor(100,103)(0,103){$\sigma_v$}[\atright,\solidarrow]
\mor(-2,100)(48,50){$\mathbf{q}$}[\atright,\solidarrow]
\mor(-2,2)(48,52){$\mathbf{q}$}[\atleft,\solidarrow]
\mor(52,48)(2,-2){$\sigma_q$}[\atleft,\solidarrow]
\mor(150,100)(100,50){$\qmapsymbol$}[\atleft,\solidarrow]
\mor(150,2)(100,52){$\qmapsymbol$}[\atright,\solidarrow]
\mor(96,48)(146,-2){$\vmapsymbol$}[\atright,\dasharrow]
\enddc}
\caption{A commutative diagram of mappings}
\label{fig-mappings}
\end{figure}

\begin{lemma} \label{l-mappings}
    Consider Figure~\ref{fig-mappings}, which shows the interactions
    between the maps $\mathbf{v}$, $\mathbf{q}$, $\sigma_v$, $\sigma_q$,
    $\kappa_s$, $\kappa_v$, $\qmapsymbol$ and $\vmapsymbol$.  Note that
    some of these maps appear twice---once in their full generality,
    and once when restricted to canonical surfaces or vectors.
    All of the unnamed hooked arrows in this diagram are inclusion maps.
    Then the following facts are true:
    \begin{enumerate}[(i)]
        \item Figure~\ref{fig-mappings} is a commutative diagram.

        \item All double arrows in this diagram represent inverse functions.
        This includes the pair
        $\mathbf{v},\sigma_v\co\surfaces\rightleftharpoons\vspaceza$,
        their canonical restrictions
        $\mathbf{v},\sigma_v\co\surfacescan\rightleftharpoons\vspacezac$,
        the pair
        $\mathbf{q},\sigma_q\co\surfacescan\rightleftharpoons\qspaceza$,
        and the pair
        $\qmapsymbol,\vmapsymbol\co\vspaceac\rightleftharpoons\qspacea$.

        \item Of the three vector-to-vector maps ($\qmapsymbol$, $\vmapsymbol$
        and $\kappa_v$), only $\qmapsymbol$ is linear.\footnote{By
        ``linear'', we only require here that
        $\qmap{\lambda\mathbf{x}+\mu\mathbf{y}}=
            \lambda\qmap{\mathbf{x}}+\mu\qmap{\mathbf{y}}$ for
        $\lambda,\mu \geq 0$.  This is because the domains
        $\vspacea$ and $\vspaceac$ are not closed under multiplication
        by $\lambda<0$.}
        The remaining maps $\vmapsymbol$ and $\kappa_v$ preserve
        scalar multiplication
        (that is, $\vmap{\lambda\mathbf{w}}=\lambda\vmap{\mathbf{w}}$
        and $\kappa_v(\lambda\mathbf{w})=\lambda\kappa_v(\mathbf{w})$
        for $\lambda \geq 0$), but they need not preserve addition.
        The non-linear maps $\vmapsymbol$ and $\kappa_v$
        are drawn in the diagram with dotted lines.
    \end{enumerate}
\end{lemma}

\begin{proof}
    These observations are all straightforward consequences
    of the relevant definitions, and we do not recount the details here.
    The one additional observation required is that vertex linking
    surfaces only contain triangular discs, which is why
    $\mathbf{q} \circ \kappa_s = \mathbf{q}$
    and $\qmapsymbol \circ \kappa_v = \qmapsymbol$
    (since $\mathbf{q}$ and $\qmapsymbol$ ignore triangular discs entirely).
\end{proof}

Note that some of the maps described by Lemma~\ref{l-mappings}
are more general than
Figure~\ref{fig-mappings} indicates.  In particular, both $\qmapsymbol$ and
$\kappa_v$ are defined on all $7n$--dimensional vectors, admissible or not.
The commutative relationship $\qmapsymbol \circ \kappa_v = \qmapsymbol$
still holds in this more general setting, but we do not worry about this here.

We return now to the two key conversion functions: the quadrilateral
projection $\qmapsymbol\co\R^{7n}\to\R^{3n}$ and the canonical extension
$\vmapsymbol\co\qspacea\to\vspaceac$.  It is clear how to compute
$\qmap{\mathbf{w}}$ quickly (just drop all triangular coordinates from
$\mathbf{w}$), but it is less clear how to compute $\vmap{\mathbf{w}}$
quickly.

A simple algorithm for computing $\vmap{\mathbf{w}}$
might run as follows.  Given a quadrilateral
vector $\mathbf{w}\in\qspacea$,
we solve the standard matching equations
using typical methods of linear algebra to obtain a matching set of
triangular coordinates (there will be many solutions but any one will do),
and then we apply $\kappa_v$ to make the resulting vector in $\R^{7n}$
canonical.

However, this algorithm is slow---to solve the standard
matching equations requires $O(n^3)$ time for a simple implementation,
though more sophisticated solvers can improve upon
this a little.\footnote{We can improve upon $O(n^3)$ by exploiting
the sparseness and rationality of the standard matching equations;
see for instance the $\simeq O(n^{2.5})$ iterative algorithm of
Eberly et al.~\cite{eberly06-linear}.}
It turns out that for the specific problem of computing $\vmap{\mathbf{w}}$
we can do much better, as seen in the following result.

\begin{algorithm} \label{a-quaddfs}
    Let $\mathbf{w} \in \qspacea$ be any admissible vector in
    quadrilateral coordinates.  Then the following algorithm computes
    the canonical extension $\vmap{\mathbf{w}}$, and does so in
    $O(n)$ time.

    We begin by constructing a vector
    \begin{alignat*}{2}
    \mathbf{x}~=\,(~
        & t_{1,1},t_{1,2},t_{1,3},t_{1,4},\ q_{1,1},q_{1,2}&&,q_{1,3}\ ;\\
        & t_{2,1},t_{2,2},t_{2,3},t_{2,4},\ q_{2,1},q_{2,2}&&,q_{2,3}\ ;\\
        & \ldots &&,q_{n,3}\ ) \in \R^{7n}
    \end{alignat*}
    whose quadrilateral coordinates $q_{i,j}$
    are copied directly from $\mathbf{w}$,
    and whose triangular coordinates $t_{i,j}$ are initially unknown.
    Then, for each
    vertex $V$ of the triangulation $\tri$, we perform the following steps:
    \begin{enumerate}[1.]
        \item \label{en-quaddfs-start}
        Choose an arbitrary triangular disc type surrounding $V$, and
        set the corresponding triangular coordinate of $\mathbf{x}$ to
        zero.

        \item \label{en-quaddfs-search}
        Run through all triangular disc types surrounding $V$ using a
        depth-first search, beginning at the disc type chosen in
        step~\enref{en-quaddfs-start} above.  By ``depth-first search'',
        we mean that after visiting some triangular disc type,
        we recursively visit the three adjacent\footnote{{\em Adjacent}
        in the sense of the standard matching equations: two adjacent
        disc types sit within adjacent tetrahedra, and their boundary
        arcs within the common tetrahedron face are parallel.  Refer
        to Figure~\ref{fig-matchingstd} for an illustration.}
        triangular disc types
        in turn (ignoring those that have been visited already).

        Each time we visit a triangular disc type, we set the
        corresponding triangular coordinate of $\mathbf{x}$ as follows.
        Suppose we are visiting the triangular disc type corresponding
        to coordinate $t_{i,a}$, having just come from the (adjacent)
        triangular disc type corresponding to coordinate $t_{j,c}$.
        Then one of the standard matching equations for $\tri$ is of the
        form $t_{i,a}+q_{i,b}=t_{j,c}+q_{j,d}$.  Since we already have
        values for $t_{j,c}$, $q_{i,b}$ and $q_{j,d}$, we can use this
        matching equation to set the unknown coordinate $t_{i,a}$
        accordingly.

        \item \label{en-quaddfs-canonise}
        Once this depth-first search is complete, we have values assigned
        to all triangular coordinates of $\mathbf{x}$ surrounding $V$.
        Let $\lambda$ be the minimum of these triangular coordinates;
        we now subtract $\lambda \ell(V)$ from $\mathbf{x}$.
    \end{enumerate}
\end{algorithm}

\begin{proof}
    First we note that the algorithm is well-defined; in particular,
    that each depth-first search in step~\enref{en-quaddfs-search}
    runs to completion (that is, we visit every triangular
    disc surrounding $V$).  This follows
    immediately from the fact that each vertex link $\ell(V)$ is connected.

    Our next task is to prove the algorithm correct.
    Consider step~\enref{en-quaddfs-start}, where we set an arbitrary
    triangular coordinate surrounding vertex $V$ to zero.
    Suppose instead that we set this coordinate to $\mu$.  By
    examining the form of the standard matching equations in
    step~\enref{en-quaddfs-search}, we see that this $+\mu$ would
    propagate through every triangular disc type surrounding $V$; in other
    words, by the end of step~\enref{en-quaddfs-search} we would have
    added an extra $\mu\ell(V)$ to the solution $\mathbf{x}$.
    However, this would then cause us to subtract an extra
    $\mu\ell(V)$ from $\mathbf{x}$ in step~\enref{en-quaddfs-canonise}.
    Therefore {\em the value given to the first triangular coordinate
    in step~\enref{en-quaddfs-start} does not affect the final solution
    $\mathbf{x}$}.

    Since $\vmap{\mathbf{w}}$ is known to satisfy the standard matching
    equations, and since the only coordinate assignment in our algorithm that
    does {\em not} use the standard matching equations
    (step~\enref{en-quaddfs-start}) turns out to be irrelevant,
    it follows that $\mathbf{x}=\vmap{\mathbf{w}}$.  That is, the
    algorithm is correct.

    Finally, we observe that the algorithm runs in $O(n)$ time.  Each
    of the $3n$ triangular disc types in $\tri$ is visited precisely once in
    steps~\enref{en-quaddfs-start} and~\enref{en-quaddfs-search}; moreover,
    for each disc type there is a small constant number of adjacencies
    (three) to examine.  It follows that, assuming we are careful with our
    implementation\footnote{For instance, when visiting a disc type in
    step~\enref{en-quaddfs-search}, we do
    not search through all other disc types to find which are adjacent;
    instead we compute this information directly in constant time.
    Likewise, we do not run through all disc types in $\tri$
    for steps~\enref{en-quaddfs-start} and~\enref{en-quaddfs-canonise}
    when we only require those surrounding a single vertex $V$.},
    the time complexity of this algorithm is indeed $O(n)$.
\end{proof}

As a final observation, $\vmapsymbol$ must construct a vector of
length $7n$ by definition, which means that {\em any} algorithm for computing
$\vmap{\mathbf{w}}$ must run in at least $O(n)$ time.
Therefore the $O(n)$ time complexity of
Algorithm~\ref{a-quaddfs} is the fastest time complexity possible.

%% file: stdtoquad.tex
\section{The Easy Direction: Standard to Quadrilateral} \label{s-stdtoquad}

At this point we are ready to build algorithms for converting between
the standard and quadrilateral {\em solution sets}.  In this section we consider
the simpler direction: converting the standard solution set into
the quadrilateral solution set.

We begin by proving some necessary and sufficient conditions for
vertex normal surfaces (Lemmata~\ref{l-vertexsplit}
and~\ref{l-vertexpositions}).  We then show that the canonical part of
every {\em quadrilateral}
vertex normal surface is also a {\em standard} vertex normal surface
(Lemma~\ref{l-quadisstd}), and use this as the basis for our
standard-to-quadrilateral conversion algorithm
(Algorithm~\ref{a-stdtoquad}).

Once again, we assume throughout this section that we are working with a
compact $3$--manifold triangulation $\tri$ built from $n$ tetrahedra.

\begin{lemma} \label{l-vertexsplit}
    Let $S$ be an embedded normal surface in $\tri$ for which
    $\vrep{S} \neq \mathbf{0}$.
    If $S$ is a standard vertex normal surface, then whenever
    $\vrep{S} = \alpha\,\mathbf{u} + \beta\,\mathbf{w}$ for
    admissible vectors $\mathbf{u},\mathbf{w}\in\R^{7n}$
    and constants $\alpha,\beta > 0$, it must be true that both
    $\mathbf{u}$ and $\mathbf{w}$ are multiples of $\vrep{S}$.
    Conversely, if $S$ is not a standard vertex normal surface,
    then there exist embedded normal surfaces $U$ and $W$ and
    rationals $\alpha,\beta > 0$ for which
    $\vrep{S} = \alpha\,\vrep{U} + \beta\,\vrep{W}$ but where
    neither $\vrep{U}$ nor $\vrep{W}$ are multiples of $\vrep{S}$.

    Moreover, these statements are also true in quadrilateral
    coordinates, where we replace ``standard'', $\vrep{\cdot}$ and
    $\R^{7n}$ with ``quadrilateral'', $\qrep{\cdot}$
    and $\R^{3n}$ respectively.
\end{lemma}

In essence, we are taking a basic fact about polytope vertices
and showing that it holds true even when we restrict our attention to
{\em admissible} vectors within the polytope.  Note that the two
statements of this lemma are not exactly converse; instead each is a
little stronger than the converse of the other, making them slightly
easier to exploit later on.

\begin{proof}
    The proofs are identical in standard and quadrilateral coordinates;
    here we consider standard coordinates only.

    Suppose $S$ is a standard vertex normal surface.  Then the given
    condition on $\mathbf{u}$ and $\mathbf{w}$ follows immediately from the
    fact that $\vproj{S}$ is a vertex of the polytope $\stdproj$.

    On the other hand, suppose that $S$ is not a standard vertex normal
    surface.  Then $\vproj{S}$ is not a vertex of the polytope $\stdproj$,
    and so we can find rational vectors $\mathbf{u},\mathbf{w} \in \stdproj$
    on opposite sides of $\vproj{S}$; that is,
    $\mathbf{u},\mathbf{w} \neq \vproj{S}$ and
    $\frac12(\mathbf{u}+\mathbf{w})=\vproj{S}$.

    We show that both $\mathbf{u}$ and $\mathbf{w}$ satisfy the
    quadrilateral constraints as follows.  Without loss of generality,
    suppose that $\mathbf{u}$ does {\em not} satisfy the quadrilateral
    constraints.  Then, since $\vproj{S}$ does, there must be some
    quadrilateral coordinate $q_{i,j}$ that is zero in $\vproj{S}$ but
    strictly positive in $\mathbf{u}$.  It follows that this coordinate
    is negative in $\mathbf{w}$, contradicting the claim that
    $\mathbf{w} \in \stdproj$ (recall that $\stdproj$ lies in the
    non-negative orthant).

    Therefore both $\mathbf{u}$ and $\mathbf{w}$ are rational vectors in
    $\stdproj$ that satisfy the quadrilateral constraints.  It follows
    from Theorem~\ref{t-admissible} that we can find embedded normal
    surfaces $U$ and $W$ for which $\vproj{U}=\mathbf{u}$ and
    $\vproj{W}=\mathbf{w}$, whereupon we find that
    $\vrep{S} = \alpha\,\vrep{U} + \beta\,\vrep{W}$ for $\alpha,\beta > 0$
    but neither $\vrep{U}$ nor $\vrep{W}$ is a multiple of $\vrep{S}$.
\end{proof}

Note that Lemma~\ref{l-vertexsplit} has slightly different implications in
standard and quadrilateral coordinates.  For instance, the condition
$\vrep{S} \neq \mathbf{0}$ requires the surface $S$ to be non-empty, but
$\qrep{S} \neq \mathbf{0}$ requires that $S$ is not a union of vertex links.
Other differences arise regarding scalar multiplication.  For example,
for certain types of two-sided surface $S$, we have that
$\vrep{U}$ is an integer multiple of $\vrep{S}$ if and only if the
surface $U$ consists of zero or more copies of $S$.  On the other hand,
$\qrep{U}$ is an integer multiple of $\qrep{S}$ if and only if $U$
consists of zero or more copies of $S$ with possibly some vertex links
added or subtracted.

Our next result allows us to identify vertex normal surfaces based
purely on which coordinates are zero and which are non-zero.

\begin{defn}[Domination]
    Let $\mathbf{x}$ and $\mathbf{y}$ be vectors in $\R^d$.
    We say that $\mathbf{x}$ {\em dominates} $\mathbf{y}$ if,
    whenever a coordinate $x_i$ is zero, the corresponding coordinate
    $y_i$ is zero also.
    We say that $\mathbf{x}$ {\em strictly dominates} $\mathbf{y}$ if
    (i) $\mathbf{x}$ dominates $\mathbf{y}$, and (ii) there is some
    coordinate $y_i$ that is zero for which the corresponding coordinate
    $x_i$ is non-zero.
\end{defn}

For instance, in $\R^3$ the vector $(0,5,3)$ strictly dominates $(0,2,0)$,
the vectors $(1,0,2)$ and $(3,0,1)$ both dominate each other (but
not strictly), and neither of $(0,2,5)$ or $(7,0,4)$ dominates the other.

When discussing domination we use $\mathbf{x}$ and $\proj{\mathbf{x}}$
interchangeably, since both $\mathbf{x}$ and $\proj{\mathbf{x}}$ have zero
coordinates in the same positions.

\begin{lemma} \label{l-vertexpositions}
    Let $S$ be an embedded normal surface in $\tri$ for which
    $\vrep{S} \neq \mathbf{0}$.
    If $S$ is a standard vertex normal surface,
    then whenever $\vrep{S}$ dominates $\mathbf{u}$ for some
    admissible vector $\mathbf{u}\in\R^{7n}$, it must be true that
    $\mathbf{u}$ is a multiple of $\vrep{S}$.
    Conversely, if $S$ is not a standard vertex normal surface,
    then there is some standard vertex normal surface $U$ for which
    $\vrep{S}$ strictly dominates $\vrep{U}$.

    Moreover, these statements are also true in quadrilateral
    coordinates, where we replace ``standard'', $\vrep{\cdot}$
    and $\R^{7n}$ with ``quadrilateral'', $\qrep{\cdot}$
    and $\R^{3n}$ respectively.
\end{lemma}

As in Lemma~\ref{l-vertexsplit}, each half of this lemma is a stronger
version of the converse of the other.  While this makes the statement of
the lemma a little less transparent, it also makes both halves easier to
use in practice (as we will see later in this section).

\begin{proof}
    Again the proofs in standard and quadrilateral coordinates are
    identical; here we consider only standard coordinates.

    Suppose that $S$ is a standard vertex normal surface and that
    $\vrep{S}$ dominates $\mathbf{u}$ for some admissible
    $\mathbf{u}\in\R^{7n}$.  If $\mathbf{u}=\mathbf{0}$ then
    $\mathbf{u}$ is clearly a multiple of $\vrep{S}$, so assume that
    $\mathbf{u}\neq\mathbf{0}$.
    Let $\mathbf{w} = \vrep{S} + \epsilon (\vrep{S} - \mathbf{u})$
    for some small $\epsilon > 0$;
    that is, $\mathbf{w}$ is an extension of the line joining
    $\mathbf{u}$ and $\vrep{S}$, just beyond $\vrep{S}$.

    Because $\vrep{S}$ and $\mathbf{u}$ satisfy the standard matching
    equations, so does $\mathbf{w}$.  Because $\vrep{S}$ dominates
    $\mathbf{u}$, we can keep the coordinates of $\mathbf{w}$ non-negative
    by choosing $\epsilon$ sufficiently small.  Finally, because
    $\vrep{S}$ satisfies the quadrilateral constraints and
    $\mathbf{u}$ introduces no new non-zero coordinates, it follows
    that $\mathbf{w}$ satisfies the quadrilateral constraints also.
    Therefore $\mathbf{w}$ is an admissible vector.
    Since $(1 + \epsilon) \vrep{S} = \mathbf{w} + \epsilon \mathbf{u}$,
    we have from Lemma~\ref{l-vertexsplit} that $\mathbf{u}$ is a multiple of
    $\vrep{S}$.

    Now suppose that $S$ is not a standard vertex normal surface.
    Let $F$ be the minimal-dimensional face of the polytope $\stdproj$
    containing $\vproj{S}$, and let $\mathbf{u}$ be any vertex of $F$.
    We aim to show that $\mathbf{u} = \vproj{U}$ for some standard
    vertex normal surface $U$,
    and that $\vproj{S}$ strictly dominates $\mathbf{u}$.

    Consider any coordinate that is zero in $\vproj{S}$; without loss of
    generality let this be $q_{i,j}$ (though it could equally well be a
    triangular coordinate).  The hyperplane $q_{i,j}=0$ is a supporting
    hyperplane for $\stdproj$, and since it contains $\vproj{S}$ it must
    contain the entire minimal-dimensional face $F$.  Therefore the
    coordinate $q_{i,j}$ is zero at every vertex of $F$, including
    $\mathbf{u}$.

    Running through all such coordinates, we see that
    $\mathbf{u}$ is dominated by $\vproj{S}$; this domination
    also shows that $\mathbf{u}$ satisfies the quadrilateral constraints.
    Since our polytope is rational and $\mathbf{u}$ is a vertex
    it follows that $\mathbf{u}=\vproj{U}$ for some standard vertex normal
    surface $U$.

    Finally, because $S$ is not a standard vertex normal surface
    we have $\vproj{S} \neq \mathbf{u}$; the first part of this lemma
    then shows that $\mathbf{u}$ cannot dominate $\vproj{S}$,
    which means that $\mathbf{u}$ must be {\em strictly} dominated
    by $\vproj{S}$.
\end{proof}

One simple but useful consequence of Lemma~\ref{l-vertexpositions}
is the following.

\begin{corollary} \label{c-vtxiscanonical}
    Every standard vertex normal surface in $\tri$ is either
    (i)~canonical, or (ii)~consists of one or more copies of the link of
    a single vertex of $\tri$.  Moreover,
    the link of a single vertex of $\tri$ is always a standard vertex
    normal surface.
\end{corollary}

\begin{proof}
    Let $S$ be a standard vertex normal surface in $\tri$, and suppose
    that $S$ is not canonical.  Then $S$ contains at least one vertex
    linking component; let this be the link $\ell(V)$.  It follows
    that $\vrep{S}=\ell(V)+\mathbf{u}$ for some non-negative
    $\mathbf{u}\in\R^{7n}$.  Thus $\vrep{S}$ dominates
    $\ell(V)$, and from Lemma~\ref{l-vertexpositions} we have that
    $\vrep{S}$ is a multiple of the vertex link $\ell(V)$.

    Now consider a single vertex link $\ell(V)$.  If this vertex link is
    not a standard vertex normal surface, then from
    Lemma~\ref{l-vertexpositions} there is some non-empty
    embedded normal surface
    $U$ for which $\ell(V)$ strictly dominates $\vrep{U}$.  Thus the
    surface $U$ contains only triangular discs surrounding the vertex
    $V$, and moreover at least one such triangular disc type does not
    appear in $U$ at all.

    By following the standard matching equations around the vertex $V$
    we find that, because {\em some} triangular coordinate surrounding $V$
    is zero in $\vrep{U}$, then {\em all} such coordinates must be zero in
    $\vrep{U}$.  Thus $U$ is the empty surface, giving a contradiction.
    %
\end{proof}

We proceed now to the key result that underpins the
standard-to-quadrilateral conversion algorithm.

\begin{lemma} \label{l-quadisstd}
    The canonical part of every quadrilateral vertex normal surface
    in $\tri$ is also a standard vertex normal surface in $\tri$.
\end{lemma}

\begin{proof}
    Let $S$ be a quadrilateral vertex normal surface, and suppose
    that the canonical part $\kappa_s(S)$ is {\em not} a standard vertex
    normal surface.  Then from
    Lemma~\ref{l-vertexsplit}, there exist embedded normal surfaces
    $U$ and $W$ where
    $\vrep{\kappa_s(S)} = \alpha\,\vrep{U} + \beta\,\vrep{W}$ for
    $\alpha,\beta > 0$ and where neither $\vrep{U}$ nor $\vrep{W}$
    is a rational multiple of $\vrep{\kappa_s(S)}$.
    Because $\kappa_s(S)$ is canonical,
    it follows from Lemma~\ref{l-canonical-linear} that both $U$ and $W$ are
    canonical also.

    Using the fact that the quadrilateral projection $\pi$ is linear
    and that $\qmapsymbol \cdot \mathbf{v} = \mathbf{q}$
    (Lemma~\ref{l-mappings}), it follows that the analogous relationship
    $\qrep{\kappa_s(S)} = \alpha\,\qrep{U} + \beta\,\qrep{W}$
    must hold in quadrilateral coordinates.  Since
    $\mathbf{q} \cdot \kappa_s = \mathbf{q}$, this simplifies to
    $\qrep{S} = \alpha\,\qrep{U} + \beta\,\qrep{W}$.

    Finally, because $S$ is a
    quadrilateral vertex normal surface, Lemma~\ref{l-vertexsplit}
    shows that both $\qrep{U}$ and $\qrep{W}$ must be rational multiples of
    $\qrep{S}=\qrep{\kappa_s(S)}$.
    Since the canonical extension $\vmapsymbol$ preserves scalar multiplication
    and $\vmapsymbol \cdot \mathbf{q} = \mathbf{v}$ on canonical
    surfaces (Lemma~\ref{l-mappings} again), this implies that
    both $\vrep{U}$ and $\vrep{W}$ are rational multiples of
    $\vrep{\kappa(S)}$, a contradiction.
\end{proof}

We close this section with our first algorithm for converting between
solution sets: the conversion from the standard solution set
to the quadrilateral solution set.  This is the easier direction in all
respects---the algorithm is conceptually simple (we use
Lemma~\ref{l-quadisstd} to find potential solutions and
Lemma~\ref{l-vertexpositions} to verify them),
it is simple to implement, and it has a guaranteed small polynomial
running time\footnote{Of course this must be polynomial in not just $n$ but
also the size of the input, i.e., the standard solution set.  There are
families of triangulations for which the standard solution set is known to
have size exponential in $n$; see \cite{burton09-extreme} for some examples.}
(which is unusual for vertex enumeration problems).

\begin{algorithm} \label{a-stdtoquad}
    Suppose we are given the standard solution set
    for the triangulation $\tri$, and that this standard solution
    set consists of the $k$ vectors
    $\mathbf{v}_1,\ldots,\mathbf{v}_k \in \R^{7n}$.
    Then the following algorithm computes the quadrilateral solution
    set for $\tri$, and does so in $O(n k^2)$ time.
    \begin{enumerate}[1.]
        \item \label{en-stdtoquad-proj}
        Compute the quadrilateral projections
        $\qmap{\mathbf{v}_1},\ldots,\qmap{\mathbf{v}_k}$;
        recall that this merely involves removing the triangular
        coordinates from each vector.  Throw away any zero vectors
        that result, and label the remaining non-zero vectors
        $\mathbf{q}_1,\ldots,\mathbf{q}_{k'} \in \R^{3n}$.

        \item \label{en-stdtoquad-pairs}
        Begin with an empty list of vectors $L$.
        For each $i=1,\ldots,k'$, test whether the vector
        $\mathbf{q}_i$ dominates any other $\mathbf{q}_j$ for $i \neq j$.
        If not, insert the projective image $\proj{\mathbf{q}_i}$ into
        the list $L$.

        \item \label{en-stdtoquad-done}
        Once step~\enref{en-stdtoquad-pairs} is complete, the list
        $L$ holds the complete quadrilateral solution set for $\tri$.
    \end{enumerate}
\end{algorithm}

\begin{proof}
    Our first task is to prove the algorithm correct.  We approach this by
    (i)~showing that every member of the quadrilateral solution
    set does appear in the final list $L$, and then (ii)~showing that any
    other vector does not appear in the final list $L$.
    \begin{itemize}
        \item
        Suppose $\mathbf{w}\in\R^{3n}$ is a member of the quadrilateral
        solution set for $\tri$.  Then $\mathbf{w}$ is non-zero, and
        furthermore
        $\mathbf{w}=\qproj{S}=\qproj{\kappa_s(S)}$ for some
        quadrilateral vertex normal surface $S$.  From Lemma~\ref{l-quadisstd},
        $\kappa_s(S)$ is also a {\em standard} vertex normal surface, and so
        $\vproj{\kappa_s(S)}$ is a member of the standard solution set.
        Therefore $\vproj{\kappa_s(S)}=\mathbf{v}_i$ for some $i$, whereupon
        Lemma~\ref{l-mappings} gives us
        $\mathbf{w}=\qproj{\kappa_s(S)}=\proj{\qmap{\vrep{\kappa_s(S)}}}
        =\proj{\qmap{\mathbf{v}_i}}$.  That is, $\mathbf{w}$ appears in
        step~\enref{en-stdtoquad-proj} as
        $\mathbf{w}=\mathbf{q}_{i'}$ for some $i'$.

        Suppose now that $\mathbf{w}$ does not appear in the final
        list $L$.  This can only be because
        $\mathbf{q}_{i'}$ dominates
        $\mathbf{q}_{j'}$ for some $j' \neq i'$.  From
        step~\enref{en-stdtoquad-proj} we know
        that $\mathbf{q}_{j'}=\qmap{\mathbf{v}_j}$
        for some vector $\mathbf{v}_j \neq \mathbf{v}_i$ in the standard
        solution set.  Moreover, neither $\mathbf{v}_i$ nor
        $\mathbf{v}_j$ is a multiple of a vertex link (otherwise
        $\mathbf{q}_{i'}$ or $\mathbf{q}_{j'}$ would be zero); therefore
        Corollary~\ref{c-vtxiscanonical} shows that both
        $\mathbf{v}_i$ and $\mathbf{v}_j$ are canonical, and so
        $\mathbf{v}_i=\vmap{\mathbf{q}_{i'}}$ and
        $\mathbf{v}_j=\vmap{\mathbf{q}_{j'}}$.

        Because $\mathbf{q}_{i'}$ dominates $\mathbf{q}_{j'}$, it follows
        from Lemma~\ref{l-vertexpositions} that $\mathbf{q}_{j'}$ is a
        multiple of $\mathbf{q}_{i'}$.  Since $\vmapsymbol$ preserves scalar
        multiplication,
        $\mathbf{v}_i=\vmap{\mathbf{q}_{i'}}$ is also a multiple of
        $\mathbf{v}_j=\vmap{\mathbf{q}_{j'}}$.  Finally, since
        $\mathbf{v}_i$ and $\mathbf{v}_j$ both belong to the standard
        solution set, their coordinates must both sum to one and we obtain
        $\mathbf{v}_i=\mathbf{v}_j$, a contradiction.

        \item
        Suppose now that $\mathbf{w}\in\R^{3n}$ is {\em not} a member of the
        quadrilateral solution set for $\tri$.
        From Lemma~\ref{l-vertexpositions} there is some quadrilateral vertex
        normal surface $U$ for which $\mathbf{w}$ strictly dominates
        $\qrep{U}$, and from the previous argument the projective image
        $\qproj{U}$ appears in step~\enref{en-stdtoquad-proj}
        as some $\mathbf{q}_{j'}$.  This domination ensures that
        $\mathbf{w}$ is tossed away in
        step~\enref{en-stdtoquad-pairs}, and so
        does not appear in the final list $L$.
    \end{itemize}

    We see then
    that $L$ contains precisely the quadrilateral solution
    set for $\tri$ as claimed.  Note
    that step~\enref{en-stdtoquad-pairs} ensures that $L$ contains no
    duplicate vectors (i.e., that $L$ is a ``true set''); otherwise each would
    dominate the other.  We finish by observing that all vector operations
    take $O(n)$ time and that
    steps~\enref{en-stdtoquad-proj} and~\enref{en-stdtoquad-pairs} require
    $O(k)$ and $O(k^2)$ vector operations respectively, giving a running
    time of $O(nk^2)$ in total.
\end{proof}

%% file: quadtostd.tex
\section{The Hard Direction: Quadrilateral to Standard} \label{s-quadtostd}

We come now to our second conversion algorithm for solution sets:
the conversion from the quadrilateral solution set to the standard
solution set.  Although this is the more difficult conversion, with a
messy implementation and a worst-case exponential running time,
it is ultimately the more useful.  In particular:
\begin{itemize}
    \item It gives us genuinely new surfaces, which
    Lemma~\ref{l-quadisstd} shows is not true in the reverse direction.
    This means that we can potentially learn new information
    about the underlying triangulation and $3$--manifold.
    \item It forms the basis for a new
    {\em enumeration} algorithm to generate the standard solution set,
    which runs orders of magnitude faster than the current state-of-the-art.
\end{itemize}

We begin with some prerequisite tools in Section~\ref{s-quadtostd-cones},
where we introduce some additional vector maps and then
discuss polyhedral cones and their interaction with the quadrilateral
constraints.
Following this, Section~\ref{s-quadtostd-algorithm} is devoted to
presenting and proving the quadrilateral-to-standard solution set
conversion algorithm (Algorithm~\ref{a-quadtostd}).
We finish in Section~\ref{s-quadtostd-consequences}
with a brief discussion of time complexity
(Conjecture~\ref{cj-complexity}) and
the new enumeration algorithm described above (Algorithm~\ref{a-enumstd}).
As discussed back in the introduction, this final enumeration
algorithm is the real ``end product'' of this paper, and we devote all of
Section~\ref{s-expt} to testing its performance in a practical setting.

As before, we assume throughout this section that we are working with a
compact $3$--manifold triangulation $\tri$ built from $n$ tetrahedra.

\subsection{Vector Maps and Polyhedral Cones}
\label{s-quadtostd-cones}

To present and prove the quadrilateral-to-standard conversion algorithm
(Algo\-rithm~\ref{a-quadtostd}), we need to call upon two new families
of vector maps, both of which involve the vertices of the
triangulation~$\tri$.

\begin{defn}[Partial Canonical Part] \label{d-partial-canonical}
    Let the vertices of $\tri$ be labelled
    $V_1,\ldots,V_m$, and let $\mathbf{w}$ be any vector in $\R^{7n}$.
    For each $i=1,\ldots,m$, the {\em $i$th partial
    canonical part} of $\mathbf{w}$ is denoted $\canp{i}{\mathbf{w}}$
    and is defined as follows.
    Let $\lambda \in \R$ be the largest scalar for which
    all of the coordinates of $\mathbf{w} - \lambda \ell(V_i)$ that
    correspond to triangular disc types surrounding $V_i$ are non-negative.
    Then we define $\canp{i}{\mathbf{w}} = \mathbf{w} - \lambda \ell(V_i)$.
\end{defn}

Essentially $\canp{i}{\mathbf{w}}$ is a ``restricted'' canonical part of
$\mathbf{w}$ where we only allow copies of the vertex link $\ell(V_i)$ to
be added or subtracted.  It is simple to see that applying this
procedure to all vertices gives the usual canonical part $\kappa_v$,
that is, $\kappa_v = \kappa_v^{(1)} \circ \ldots \circ \kappa_v^{(m)}$.
Like $\kappa_v$, the partial maps $\kappa_v^{(i)}$ are not linear but do
preserve scalar multiplication.

\begin{defn}[Truncation] \label{d-truncation}
    Let the vertices of $\tri$ be labelled
    $V_1,\ldots,V_m$, and let $\mathbf{w}$ be any vector in $\R^{7n}$.
    For each $i=0,\ldots,m$, the {\em $i$th truncation} of $\mathbf{w}$
    is denoted $\tau_i(\mathbf{w})$, and is defined as follows.
    We first locate all coordinates in $\mathbf{w}$ that correspond to
    triangular disc types surrounding the vertices $V_{i+1},\ldots,V_m$.
    Then $\tau_i(\mathbf{w})$ is obtained from $\mathbf{w}$ by setting
    each of these coordinates to zero.

    For convenience, if $S \subseteq \R^{7n}$ is any set of vectors then
    we let $\tau_i(S)$ denote the corresponding set of $i$th truncations;
    that is, $\tau_i(S) = \{ \tau_i(\mathbf{w})\,|\,\mathbf{w} \in S\}$.
\end{defn}

The $0$th truncation $\tau_0(\mathbf{w})$ is most severe, setting all
triangular coordinates in $\mathbf{w}$ to zero.  At the other extreme,
the $m$th truncation has no effect whatsoever, with
$\tau_m(\mathbf{w})=\mathbf{w}$.  Each truncation map is linear,
and it is clear that $\tau_i \circ \tau_j = \tau_{\min(i,j)}$.
Note that truncation does not preserve admissibility, since
$\tau_i(\mathbf{w})$ might not satisfy the standard matching equations
even if $\mathbf{w}$ does.

In general it is impossible to undo truncations precisely.
However, for admissible vectors the errors are controllable,
as seen in the following result.

\begin{lemma} \label{l-trunc-error}
    Consider any two admissible vectors $\mathbf{x},\mathbf{y} \in \R^{7n}$.
    If $\tau_{i-1}(\mathbf{x}) = \tau_{i-1}(\mathbf{y})$, then
    $\tau_i(\mathbf{x}) = \tau_i(\mathbf{y}) + \mu \ell(V_i)$ for some
    $\mu \in \R$.
\end{lemma}

\begin{proof}
    Because $\tau_{i-1}$ does not affect any quadrilateral coordinates, we
    have $\qmap{\mathbf{x}} = \qmap{\mathbf{y}}$.  With
    Lemma~\ref{l-mappings} we can convert this into
    $\kappa_v(\mathbf{x}) = \kappa_v(\mathbf{y})$,
    whereupon the result is a simple consequence of
    Definition~\ref{d-canonical-part}.
\end{proof}

For the remainder of this section we focus on polyhedral cones.
These are used heavily in the proof of Algorithm~\ref{a-quadtostd},
and we concentrate in particular on their interaction with the
quadrilateral constraints.

\begin{defn}[Polyhedral Cone] \label{d-cone}
    A {\em polyhedral cone} in $\R^d$ is an intersection of
    finitely many closed half-spaces in $\R^d$, all of whose
    bounding hyperplanes pass through the origin.

    A {\em pointed polyhedral cone} in $\R^d$ is a polyhedral cone in $\R^d$
    for which the origin is an extreme point.  Equivalently, it is a
    polyhedral cone in $\R^d$ that has a supporting
    hyperplane meeting it only at the origin.
\end{defn}

It is clear that every polyhedral cone $C$ is convex and
closed under non-negative scalar
multiplication (that is, $\mathbf{x},\mathbf{y} \in C$ implies
$\lambda \mathbf{x} + \mu \mathbf{y} \in C$ for all $\lambda,\mu \geq 0$).
An example of a polyhedral cone that is not pointed is the infinite prism
$\{\mathbf{x}\in\R^3\,|\,x_1,x_2 \geq 0\}$, for which any supporting
hyperplane containing $\mathbf{0}$ must also contain the entire line
$x_1=x_2=0$.

\begin{figure}[htb]
\centerline{\includegraphics[scale=0.6]{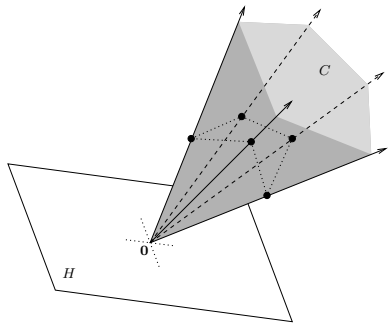}}
\caption{A pointed polyhedral cone with five basis vectors}
\label{fig-cone}
\end{figure}

\begin{defn}[Basis] \label{d-basis}
    Let $S$ be any set of vectors in $\R^d$.  By a {\em basis} for $S$,
    we mean a subset of vectors $B \subset S$ for which
    \begin{enumerate}[(i)]
        \item \label{en-basis-comb}
        every vector of $S$ can be expressed as a non-negative linear
        combination of vectors in $B$;
        \item \label{en-basis-remove}
        if any vector is removed from $B$ then
        property~(\enref{en-basis-comb}) no longer holds.
    \end{enumerate}
    It is straightforward to see that we can replace
    (\enref{en-basis-remove}) with the equivalent property
    \begin{itemize}
        \item[($\mathrm{\enref{en-basis-remove}}'$)]
        no vector in $B$ can be expressed as a non-negative
        linear combination of the others.
    \end{itemize}
\end{defn}

Although our definition of a basis is designed with polyhedral cones
in mind, it is deliberately broad; this is because we will
need to apply it not only to polyhedral cones but also to non-convex sets,
such as the {\em semi-admissible parts} to be defined shortly.
Note that for general sets $S$, property~(\enref{en-basis-comb}) does not
work in reverse---there might well be non-negative linear combinations
of vectors in $B$ that are not elements of the set $S$.

For a pointed polyhedral cone $C$, the vectors in a basis correspond to the
edges of the cone; these edges are also known as {\em extremal rays} of $C$.
Figure~\ref{fig-cone} illustrates a pointed polyhedral cone $C$ with a
supporting
hyperplane $H$ as described by Definition~\ref{d-cone}; the five points
marked in black together form a basis for $C$.  The basis for a pointed
polyhedral cone is essentially unique and can be used to reconstruct the
cone, as noted by the following well known results.

\begin{lemma} \label{l-basis-unique}
    Every polyhedral cone $C$ has a finite basis.  Moreover, if $B$ and
    $B'$ are both bases for a {\em pointed} polyhedral cone $C$,
    then there is a one-to-one
    correspondence between $B$ and $B'$ that takes each vector to a
    positive scalar multiple of itself.
\end{lemma}

\begin{lemma} \label{l-basis-construct}
    Let $B \subset \R^d$ be a finite set of vectors for which
    \begin{enumerate}[(i)]
        \item no element of $B$ can be expressed as a non-negative
        linear combination of the others;
        \item there is some hyperplane $H \subset \R^d$ passing through
        $\mathbf{0}$ for which every vector of $B$ lies strictly to the same
        side of $H$ (in particular, none of these vectors lie within $H$).
    \end{enumerate}
    Then the set of all non-negative linear combinations of vectors in
    $B$ forms a pointed polyhedral cone with $B$ as its basis.
\end{lemma}

Some pairs of basis vectors are {\em adjacent}, in the sense that the
corresponding edges of the cone are joined by two-dimensional
faces.\footnote{Note that the only one-dimensional faces of a
polyhedral cone are its extremal rays, i.e., rays of the form
$\{\lambda \mathbf{b}\,|\,\lambda > 0\}$ where $\mathbf{b}$ is a basis vector.}
In Figure~\ref{fig-cone} above, adjacent pairs of basis vectors are marked
by dotted lines.  We define adjacency formally as follows.

\begin{defn}[Adjacency]
    Let $\mathbf{b}$ and $\mathbf{b}'$ be two distinct basis vectors for a
    pointed polyhedral cone $C$.  We define $\mathbf{b}$ and $\mathbf{b}'$
    to be {\em adjacent} if the smallest-dimensional face of $C$
    containing both $\mathbf{b}$ and $\mathbf{b}'$ has dimension two.
\end{defn}

Bases of polyhedral cones provide a very limited form of uniqueness when
taking non-negative linear combinations, as seen in the following
simple lemma.

\begin{lemma} \label{l-basis-lc}
    Let $B=\{\mathbf{b}_1,\ldots,\mathbf{b}_k\}$
    be a basis for a pointed polyhedral cone $C \subset \R^d$.
    If some $\mathbf{b}_r \in B$ can be written as a non-negative linear
    combination of basis vectors (that is,
    $\mathbf{b}_r = \sum \lambda_i \mathbf{b}_i$ where all
    $\lambda_i \geq 0$), then this linear combination must be the trivial
    $\mathbf{b}_r=\mathbf{b}_r$.  That is, $\lambda_r=1$ and
    $\lambda_i=0$ for $i \neq r$.
\end{lemma}

\begin{proof}
    Suppose we have some non-negative linear combination
    $\mathbf{b}_r = \sum \lambda_i \mathbf{b}_i$.
    If $\lambda_r < 1$ then we obtain $\mathbf{b}_r$ as a non-negative
    linear combination of the other basis vectors
    $\mathbf{b}_i$~($i \neq r$), in violation of
    Definition~\ref{d-basis}.  Therefore $\lambda_r \geq 1$, and we
    can subtract $\mathbf{b}_r$ to obtain
    $\mathbf{0}$ as a non-negative linear combination
    $\mathbf{0} = \sum \lambda_i' \mathbf{b}_i$.

    Since our cone is pointed, it has a supporting hyperplane $H$ for which
    $\mathbf{0} \in H$ but every $\mathbf{b}_i$ lies strictly to one side
    of $H$.  The only way to obtain this with non-negative
    $\lambda_i'$ is to set every $\lambda_i'=0$, showing our original
    linear combination to be the trivial $\mathbf{b}_r=\mathbf{b}_r$.
\end{proof}

The uniqueness in Lemma~\ref{l-basis-lc} is limited in the sense that
it only holds when $\mathbf{b}_r$ is a basis vector.  In general,
an arbitrary point $\mathbf{x} \in C$ might well be expressible
as a non-negative linear combination of basis vectors
in several different ways.  Even for basis elements, it should be noted that
Lemma~\ref{l-basis-lc} can fail for non-pointed cones.

An even weaker form of uniqueness exists for combinations of
adjacent basis vectors, and indeed can be used to completely
characterise adjacency as follows.

\begin{lemma} \label{l-basis-adj-lincomb}
    Let $B=\{\mathbf{b}_1,\ldots,\mathbf{b}_k\}$
    be a basis for a pointed polyhedral cone $C \subset \R^d$.
    Two distinct basis vectors $\mathbf{b}_r,\mathbf{b}_s \in B$
    are adjacent if and only if, whenever
    $\mu \mathbf{b}_r + \eta \mathbf{b}_s = \sum \lambda_i \mathbf{b}_i$
    for $\mu,\eta,\lambda_i \geq 0$, we must have
    $\lambda_i=0$ for every $i \neq r,s$.
\end{lemma}

In other words, $\mathbf{b}_r$ and $\mathbf{b}_s$ are adjacent if and only
if any non-negative linear combination of basis vectors
$\mathbf{b}_r$ and $\mathbf{b}_s$ can {\em only} be expressed as a
non-negative linear combination of basis vectors
$\mathbf{b}_r$ and $\mathbf{b}_s$.

\begin{proof}
    To prove this we use two equivalent characterisations of faces
    for polyhedral cones\footnote{Although these characterisations
    are equivalent for polytopes and polyhedra, they are not equivalent
    for general convex sets.}, both of which are described by
    Br{\o}ndsted \cite{brondsted83}:
    \begin{enumerate}[(a)]
        \item \label{en-face-support}
        A set $F \subseteq C$ is a face of $C$ if and only if
        $F = C$, $F = \emptyset$, or $F = C \cap H$ for some
        supporting hyperplane $H$;
        \item \label{en-face-cross}
        A set $F \subseteq C$ is a face of $C$ if and only if
        (i)~$F$ is convex, and (ii)~whenever the open line
        segment $(\mathbf{x},\mathbf{y})$
        contains a point in $F$ for some $\mathbf{x},\mathbf{y} \in C$,
        the entire closed line segment $[\mathbf{x},\mathbf{y}]$ lies in $F$.
    \end{enumerate}
    We also note that every face of a polyhedral cone (and thus every
    supporting hyperplane above) must pass through the origin.

    Suppose the basis vectors $\mathbf{b}_r$ and $\mathbf{b}_s$ are
    adjacent, and that
    $\mu \mathbf{b}_r + \eta \mathbf{b}_s = \sum \lambda_i \mathbf{b}_i$
    for some $\mu,\eta,\lambda_i \geq 0$.
    Let $F$ be the smallest-dimensional face of $C$ containing both
    $\mathbf{b}_r$ and $\mathbf{b}_s$; since $F$ is two-dimensional,
    it cannot contain any other basis vector $\mathbf{b}_i$ for
    $i \neq r,s$.

    Using (\enref{en-face-support}) above, we can write $F = C \cap H$
    for some supporting hyperplane $H$ passing through the origin.
    We see that $\mathbf{b}_r$ and $\mathbf{b}_s$ lie in $H$ and
    every other basis vector lies strictly to one side of $H$,
    whereupon our non-negative linear combination must have
    $\lambda_i=0$ for every $i \neq r,s$.

    Suppose now that the basis vectors $\mathbf{b}_r$ and $\mathbf{b}_s$
    are not adjacent.  Let $G$ be the two-dimensional plane passing through
    $\mathbf{b}_r$, $\mathbf{b}_s$ and the origin; the non-adjacency
    of $\mathbf{b}_r$ and $\mathbf{b}_s$ shows that $G$ cannot be a face
    of $C$.  Therefore, by (\enref{en-face-cross}) above, there are points
    $\mathbf{x},\mathbf{y} \in C$ for which
    $(\mathbf{x},\mathbf{y})$ meets $G$ but
    $[\mathbf{x},\mathbf{y}] \nsubseteq G$.

    Let $\mathbf{z} \in (\mathbf{x},\mathbf{y}) \cap G$.  Because
    $\mathbf{z} \in G$ we can write
    $\mathbf{z} = \mu \mathbf{b}_r + \eta \mathbf{b}_s$ for some
    $\mu,\eta \geq 0$.  On the other hand, we can also write
    $\mathbf{z}$ as a non-trivial convex
    combination of $\mathbf{x}$ and $\mathbf{y}$.
    Since $[\mathbf{x},\mathbf{y}] \nsubseteq G$ at
    least one of $\mathbf{x}$ and $\mathbf{y}$ cannot be expressed
    purely in terms of $\mathbf{b}_r$ and $\mathbf{b}_s$, and we obtain
    $\mathbf{z} = \mu \mathbf{b}_r + \eta \mathbf{b}_s =
    \sum \lambda_i \mathbf{b}_i$ where every $\lambda_i \geq 0$
    and some $\lambda_i > 0$ for $i \neq r,s$.
\end{proof}

There are other characterisations of adjacency, such as the algebraic and
combinatorial conditions described by Fukuda and Prodon
\cite{fukuda96-doubledesc}.  However, Lemma~\ref{l-basis-adj-lincomb} will be
more useful to us when we come to the proof of Algorithm~\ref{a-quadtostd}.

The {\em double description method}, devised by Motzkin
et al.~\cite{motzkin53-dd} and improved upon by other authors since,
is a standard algorithm for inductively converting
a set of half-spaces that define a polyhedral cone into a basis for
this same cone.  The double description method plays an important role
in the standard enumeration of normal surfaces; the reader is referred
to \cite{burton08-dd} for both theoretical and practical details.
Although we do not explicitly call upon the double description method here,
we do rely on one of its core components, which is the following result.

\begin{lemma} \label{l-dd-core}
    Let $C \subset \R^d$ be a pointed polyhedral cone with basis $B$,
    and let $H$ be a half-space defined by the linear inequality
    $H = \{\mathbf{x} \in \R^d\,|\,\mathbf{x} \cdot \mathbf{h} \geq 0\}$.
    Then the intersection $C \cap H$ is also a pointed polyhedral cone,
    and we can compute a basis for $C \cap H$ as follows.

    Partition the basis $B$ into sets
    $S_0 = \{\mathbf{b} \in B\,|\,\mathbf{b} \cdot \mathbf{h} = 0\}$,
    $S_+ = \{\mathbf{b} \in B\,|\,\mathbf{b} \cdot \mathbf{h} > 0\}$ and
    $S_- = \{\mathbf{b} \in B\,|\,\mathbf{b} \cdot \mathbf{h} < 0\}$.  Then
    a basis for $C \cap H$ is
    \begin{equation*}
    S_0 \cup S_+ \cup \left\{
        \frac{(\mathbf{u}\cdot\mathbf{h})\mathbf{w} -
              (\mathbf{w}\cdot\mathbf{h})\mathbf{u}}
             {(\mathbf{u}\cdot\mathbf{h}) -
              (\mathbf{w}\cdot\mathbf{h})}
        \,\left|\,
        \begin{array}{l}
        \mathbf{u} \in S_+\ \text{and}\ \mathbf{w} \in S_-, \\
        \text{$\mathbf{u},\mathbf{w}$ are adjacent basis vectors of $C$}
        \end{array}
    \right.\right\}.
    \end{equation*}
\end{lemma}

\begin{figure}[htb]
\centerline{\includegraphics[scale=0.8]{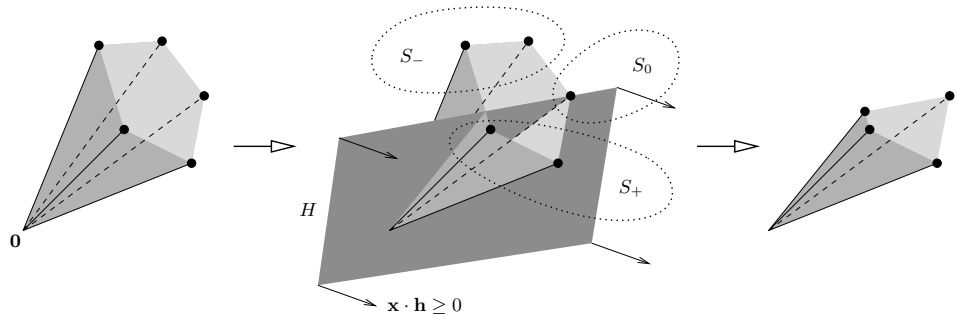}}
\caption{Intersecting a pointed polyhedral cone with a new half-space}
\label{fig-dd}
\end{figure}

This procedure is illustrated in Figure~\ref{fig-dd}.  For further details on
the double description method (including Lemma~\ref{l-dd-core}),
the reader is referred to the
excellent overview by Fukuda and Prodon \cite{fukuda96-doubledesc}.

When we come to proving Algorithm~\ref{a-quadtostd}, we will need to work
with restricted portions of polyhedral cones that satisfy the quadrilateral
constraints.  This motivates the following definition.

\begin{defn}[Semi-Admissible Part] \label{d-semi-adm}
    Consider any set of vectors $S \subseteq \R^{7n}$.  The
    {\em semi-admissible part} of $S$, denoted $\adm{S}$,
    is the subset of all vectors in $S$ that satisfy the quadrilateral
    constraints.
\end{defn}

We call this the {\em semi}-admissible part because we deliberately make
no mention of non-negativity or the matching equations.  This is
essential---in Algorithm~\ref{a-quadtostd}
we deal with vectors that satisfy the quadrilateral constraints
but that can have negative coordinates, and in the corresponding proof
we take $i$th truncations of these
vectors which can break the matching equations.

It is important to note that the semi-admissible part of a polyhedral
cone $C$ may well be non-convex, and so might not be a polyhedral cone in
itself.  Nevertheless, $\adm{C}$ remains closed under non-negative scalar
multiplication.

The following result shows that, for ``sufficiently non-negative''
pointed polyhedral cones $C$, bases for $C$ and bases for $\adm{C}$ are
tightly related.

\begin{lemma} \label{l-adm-basis}
    Let $C$ be a pointed polyhedral cone in $\R^{7n}$ where, for every
    $\mathbf{x} \in C$, the quadrilateral coordinates of $\mathbf{x}$
    are all non-negative.  If $B$ is a basis for $C$, then the
    semi-admissible part $\adm{B}$ forms a basis for $\adm{C}$.
    Conversely, every basis for $\adm{C}$ can be expressed in the form
    $\adm{B}$ where $B$ is a basis for $C$.
\end{lemma}

\begin{proof}
    Suppose that $B=\{\mathbf{b}_1,\ldots,\mathbf{b}_k\}$ is a basis for $C$.
    \begin{itemize}
        \item Since $B$ is a basis it is clear that every
        $\mathbf{x} \in \adm{C} \subseteq C$ can be expressed as
        $\mathbf{x} = \sum \lambda_i \mathbf{b}_i$ with all $\lambda_i \geq 0$.
        Furthermore, if $\lambda_i > 0$ for any
        $\mathbf{b}_i$ that does {\em not} satisfy the
        quadrilateral constraints, the non-negativity condition on $C$
        ensures that $\mathbf{x}=\lambda_i \mathbf{b}_i + \ldots$ cannot
        satisfy the quadrilateral constraints either.
        Thus $\mathbf{x} = \sum \lambda_i \mathbf{b}_i$ is actually
        a non-negative linear combination of vectors in $\adm{B}$.

        \item Since no element of $B$ can be expressed as a non-negative
        linear combination of the others, the same must be true of
        $\adm{B} \subseteq B$.
    \end{itemize}
    It follows by Definition~\ref{d-basis}
    that $\adm{B}$ is a basis for $\adm{C}$.

    Conversely, let $B'$ be a basis for $\adm{C}$, and let
    $B=\{\mathbf{b}_1,\ldots,\mathbf{b}_k\}$ be some basis for $C$.
    For each $\mathbf{b}_i \in \adm{B}$, we modify $B$ as follows.
    \begin{itemize}
        \item Since $B'$ is a basis for $\adm{C}$, we can
        express $\mathbf{b}_i$ as a non-negative linear combination of
        elements of $B'$; we mark this linear combination ($\star$)
        for later reference.  Because $B$ is a basis for $C$, we can
        expand ($\star$) to a non-negative linear combination of elements
        of $B$.  Thus we obtain $\mathbf{b}_i = \sum \lambda_j \mathbf{b}_j$
        for $\lambda_j \geq 0$.

        However, Lemma~\ref{l-basis-lc} shows that the only such linear
        combination can be $\mathbf{b}_i = \mathbf{b}_i$.  Since all
        linear combinations are non-negative, it follows that the
        first linear combination ($\star$) must likewise consist only of
        positive multiples of $\mathbf{b}_i$; in particular, we must have
        $\mu \mathbf{b}_i \in B'$ for some $\mu > 0$.
        We now replace $\mathbf{b}_i$ with $\mu \mathbf{b}_i$
        in $B$; it is clear that $B$ remains a basis for $C$.
    \end{itemize}
    By following this procedure for each $\mathbf{b}_i \in \adm{B}$,
    we obtain a basis $B$ for $C$ that satisfies
    $B' \supseteq \adm{B}$.  However, from the first part of this lemma
    $\adm{B}$ is also a basis for $\adm{C}$.  Therefore any additional
    vectors in $B'$ would be redundant, and so we have $B' = \adm{B}$.
\end{proof}

We conclude our brief study of polyhedral cones with an example of
a semi-admissible part and its basis that we have seen before.
The following observations are all immediate consequences of
the relevant definitions and Lemma~\ref{l-adm-basis}.

\begin{example} \label{ex-basis-soln}
    Let $C \subset \R^{7n}$ be the set of all vectors whose entries are
    all non-negative and which satisfy the standard matching equations for
    the triangulation $\tri$.  Then $C$ is a pointed polyhedral cone, the
    standard projective solution space $\stdproj$ is a finite cross-section
    of this cone (taken along the projective hyperplane $J^{7n}$),
    and the vertices of the polytope $\stdproj$ form a basis for $C$.
    Furthermore, $\adm{C} = \R^{7n}_a$ (the set of all admissible vectors in
    $\R^{7n}$), and the standard solution set forms a basis for $\adm{C}$.
\end{example}

\subsection{The Main Conversion Algorithm}
\label{s-quadtostd-algorithm}

We are now ready to present the quadrilateral-to-standard solution set
conversion algorithm in full detail.  The algorithm relies
on the numbering of standard coordinate positions---here
we number coordinate positions $1,2,\ldots,7n$ according to
Definition~\ref{d-vecrep}, so that positions $7i+\{1,2,3,4\}$ correspond to
triangular coordinates and positions $7i+\{5,6,0\}$ correspond
to quadrilateral coordinates.  For an arbitrary vector
$\mathbf{w}\in\R^{7n}$, we use the common notation whereby
$w_i \in \R$ denotes the coordinate of $\mathbf{w}$ in the $i$th position.

Roughly speaking, the algorithm operates as follows.  Given the $m$
vertices $V_1,\ldots,V_m$ of the triangulation, we inductively build
lists of vectors $L_0,L_1,\ldots,L_m$.  Each list $L_r$ generates
all admissible vectors that can be formed by (i)~combining vectors from the
quadrilateral solution space and then (ii)~adding {\em or subtracting}
vertex links $\ell(V_1),\ldots,\ell(V_r)$.  In particular, the initial list
$L_0$ is the quadrilateral solution set, and (after appropriate scaling)
the final list $L_m$ becomes the standard solution set.

Each inductive step that transforms $L_r$ into $L_{r+1}$ is based on the
double description method, though complications arise because we do not
have access to the full facet structures of the underlying polyhedral cones.
As we construct each list $L_r$ we essentially ignore all triangular
coordinates around the subsequent vertices
$V_{r+1},\ldots,V_m$, though we do maintain
the standard matching equations at all times.  This selective ignorance
is expressed in the proof through the truncation function $\tau_r$,
and is resolved in the algorithm itself by taking the partial canonical part
$\kappa_v^{(r)}$ when the need arises.


\begin{algorithm} \label{a-quadtostd}
    Suppose we are given the quadrilateral solution set for the
    triangulation $\tri$, and that this quadrilateral solution set
    consists of the $k$ vectors
    $\mathbf{q}_1,\ldots,\mathbf{q}_k \in \R^{3n}$.
    Then the following algorithm computes the standard solution set for $\tri$.

    Let the vertices of $\tri$ be $V_1,\ldots,V_m$.  We construct lists
    of vectors $L_0,L_1,\ldots,L_m \subset \R^{7n}$ as
    follows.\footnote{We use set notation with these lists because, as we
    see in the proof, they contain no duplicate vectors.  We call them
    lists here because the implementation can happily treat them as such;
    in particular, there is no need to explicitly check for duplicates when
    we insert vectors into lists as the algorithm progresses.}
    \begin{enumerate}[1.]
        \item \label{en-convert-ext}
        Fill the list $L_0$ with the canonical extensions
        $\vmap{\mathbf{q}_1}, \ldots, \vmap{\mathbf{q}_k} \in \R^{7n}$,
        using Algorithm~\ref{a-quaddfs} to perform the computations.

        \item \label{en-convert-pos}
        Create a set of coordinate positions
        $C \subseteq \{1,2,\ldots,7n\}$ and initialise this to
        the set of all quadrilateral coordinate positions, so that
        \[ C = \{ 5,6,7,\ 12,13,14,\ \ldots,\ 7n-2,7n-1,7n \} . \]
        This set will grow as the algorithm runs, eventually expanding to
        all of $\{1,2,\ldots,7n\}$.

        \item \label{en-convert-vtx}
        For each $r=1,2,\ldots,m$, fill the list $L_r$ as follows.
        \begin{enumerate}[(a)]
            \item \label{en-convert-vtx-init}
            For each vector $\mathbf{x} \in L_{r-1}$, insert the partial
            canonical part $\canp{r}{\mathbf{x}}$ into $L_r$.

            \item \label{en-convert-vtx-neglink}
            Insert the negative vertex link $-\ell(V_r)$ into $L_r$.

            \item \label{en-convert-vtx-dd}
            Let $T_r \subset \{1,2,\ldots,7n\}$ be the set of all
            coordinate positions corresponding to triangular disc types
            in the vertex link $\ell(V_r)$; that is,
            $T_r=\{p\,|\,\ell(V_r)_p \neq 0\}$.  For each position
            $p \in T_r$, perform the following steps.
            \begin{enumerate}[(i)]
                \renewcommand{\labelitemi}{\normalfont\bfseries \textendash}

                \item Partition the list $L_r$ into three lists
                $S_0$, $S_+$ and $S_-$ according to the sign of
                the $p$th coordinate.  Specifically, let
                $S_0 = \{ \mathbf{x} \in L_r\,|\,x_p = 0\}$,
                $S_+ = \{ \mathbf{x} \in L_r\,|\,x_p > 0\}$ and
                $S_- = \{ \mathbf{x} \in L_r\,|\,x_p < 0\}$.

                \item Create a new temporary list
                $L' = S_0 \cup S_+$.

                \item Run through all pairs of vectors
                $\mathbf{u} \in S_+$ and $\mathbf{w} \in S_-$
                that satisfy both of the following conditions:
                \begin{itemize}
                    \item[--] $\mathbf{u}$ and $\mathbf{w}$
                    together satisfy the quadrilateral constraints.
                    That is, for each tetrahedron $\Delta_i$ of $\tri$,
                    at least two of the three quadrilateral coordinates
                    for $\Delta_i$ are zero in
                    both $\mathbf{u}$ and $\mathbf{w}$ simultaneously.
                    \item[--] There is no vector $\mathbf{z} \in L_r$
                    other than $\mathbf{u}$ and $\mathbf{w}$
                    for which, whenever a coordinate position $i \in C$
                    satisfies both $u_i=0$ and $w_i=0$, then $z_i=0$ also.
                \end{itemize}
                For each such pair, insert the vector
                $(u_p \mathbf{w} - w_p \mathbf{u}) / (u_p - w_p)$
                into the temporary list $L'$.
                Note that this vector is the point where the line
                joining $\mathbf{u}$ and $\mathbf{w}$ meets the
                hyperplane $\{\mathbf{x}\in\R^{7n}\,|\,x_p=0\}$.

                \item Empty out the list $L_r$ and refill it with
                the vectors in $L'$, and insert the coordinate
                position $p$ into the set $C$.
            \end{enumerate}

            \item \label{en-convert-vtx-poslink}
            Finish the list by inserting the positive
            vertex link $\ell(V_r)$ into $L_r$.
        \end{enumerate}
    \end{enumerate}

    Suppose that the very last list $L_m$ consists of the $k'$ vectors
    $\mathbf{v}_1,\ldots,\mathbf{v}_{k'} \in \R^{7n}$.
    Then the standard solution set for $\tri$ consists of the $k'$
    projective images
    $\proj{\mathbf{v}_1},\ldots,\proj{\mathbf{v}_{k'}}$.
\end{algorithm}

Before we embark on a proof that this algorithm is correct, there are a
few points worth noting.
\begin{itemize}
    \item Unlike the previous algorithms in this paper, the statement
    of Algorithm~\ref{a-quadtostd} does not include a time complexity.
    This is because the algorithm can grow exponentially slow with
    respect to the size of the input.

    For examples of this exponential growth the reader is referred to
    \cite{burton09-extreme}, which describes the solution sets for
    the $n$--tetrahedron {\em twisted layered loop},
    a highly symmetric triangulation of
    the quotient space $S^3/Q_{4n}$.  In these examples
    the quadrilateral solution set has size $\Theta(n)$, whereas the
    standard solution set has size $\Theta(\phi^n)$ for $\phi=(1+\sqrt{5})/2$.
    Thus the size of the output is exponential in the size
    of the input, and so the running time of any conversion algorithm
    must be at least this bad.\footnote{Here we use the standard
    notation for complexity whereby $O(\cdot)$ indicates an asymptotic
    upper bound and $\Theta(\cdot)$ indicates an asymptotically tight bound.}

    On the other hand, it is possible---and indeed quite plausible---that
    the running time of Algorithm~\ref{a-quadtostd} is polynomial
    in the size of the {\em output}.  For further discussion, see
    Conjecture~\ref{cj-complexity} later in this section.

    \item Step~\enref{en-convert-vtx}(c) bears a resemblance to
    the double description method of Motzkin et al.~\cite{motzkin53-dd}.
    As discussed earlier,
    this is no accident---in a sense, within each iteration of
    step~\enref{en-convert-vtx} we create a new pointed polyhedral cone
    and then enumerate its admissible extreme rays.
    The differences appear in the processing of pairs
    $\mathbf{u} \in S_+$ and $\mathbf{w} \in S_-$,
    where we deviate from the usual double description method in the
    constraints on $\mathbf{u}$ and $\mathbf{w}$.

    \item As presented, Algorithm~\ref{a-quadtostd} requires exact arithmetic
    on rational numbers, which may be undesirable in practice for reasons of
    performance or implementation.  We can avoid this by observing that
    throughout steps~\enref{en-convert-ext}--\enref{en-convert-vtx}
    we can replace any vector $\mathbf{x}$ with any multiple
    $\lambda \mathbf{x}$ ($\lambda > 0$) without changing the final
    solution set.  This means that we can work entirely within the integers by
    rescaling vectors appropriately.
\end{itemize}


\begin{proof}[Proof of Algorithm~\ref{a-quadtostd}]
    This is a lengthy proof, consisting of two nested inductions
    corresponding to the two nested loops of steps~\enref{en-convert-vtx}
    and~\enref{en-convert-vtx}(c).
    We therefore split this proof into six parts:
    \pref{en-proof-start}--\pref{en-proof-finish} to establish the outer
    induction, and \pref{en-proof-dd-start}--\pref{en-proof-dd-finish}
    to establish the inner induction.  The road map for parts
    \pref{en-proof-start}--\pref{en-proof-finish} is given below.
    Because the inner induction sits within part~\pref{en-proof-ind},
    we delay the road map for parts
    \pref{en-proof-dd-start}--\pref{en-proof-dd-finish} until then.

    Our outer induction proves a statement about the finished lists
    $L_0,\ldots,L_m$.  In order to make this statement, we define
    the space $\ainv{r}$ for each $r=0,\ldots,m$ to be
    \begin{equation}
    \label{eqn-a-defn}
      \ainv{r}\ =\ \left\{\left.
      \mathbf{x} = \sum_{i=1}^k \lambda_i \vmap{\mathbf{q}_i} +
      \sum_{j=1}^r \mu_j \ell(V_j)\ \right|\ \begin{array}{l}
      \lambda_i \geq 0\ \forall i \in 1..k, \\
      x_p \geq 0\ \forall p \in 1..7n
      \end{array}\right\}.
    \end{equation}
    That is, $\ainv{r}$ consists of all non-negative vectors in $\R^{7n}$
    that can be expressed
    as (i)~a non-negative linear combination of the original $k$ vectors
    from the quadrilateral solution space, plus
    (ii)~arbitrary positive {\em or negative} multiples of the first
    $r$ vertex links.\footnote{It can be shown that each $\ainv{r}$ is a
    pointed polyhedral cone, though we do not need this fact here.}
    Our key inductive claim relates the space $\ainv{r}$ to the list $L_r$
    as follows:

    \displayclaim{Once the list $L_r$ is fully constructed, it
        consists only of admissible vectors.  Furthermore, the
        truncation $\tau_r(L_r)$ is a basis for the semi-admissible
        part $\adm{\tau_r(\ainv{r})}$.}

    \noindent
    Our outer induction now proceeds according to the following plan.

    \displayroadmap{\textbf{Road map for parts
        \pref{en-proof-start}--\pref{en-proof-finish}:}
        \begin{enumerate}[I.]
            \item \label{en-proof-start}
            Show that {\claima} is true for $r=0$;
            \item \label{en-proof-ind}
            Show that if {\claima} is true for $r=i-1$ where $i > 0$
            then {\claima} is also true for $r=i$;
            \item \label{en-proof-finish}
            Show that if {\claima} is true for $r=m$ then
            Algorithm~\ref{a-quadtostd} is correct.
        \end{enumerate}
    }
    Because part~\pref{en-proof-ind} is significantly more complex than the
    others (in particular, it contains the inner induction), we shall
    subvert the natural order of things and
    deal with parts~\pref{en-proof-start} and~\pref{en-proof-finish}
    first.

    \proofpart{Part~\pref{en-proof-start}}

    We begin with part~\pref{en-proof-start}, where we must prove
    {\claima} for $r=0$.  Note that $\ainv{0}$ can be written more simply as
    \[ \ainv{0}\ =\ \left\{\left.
      \mathbf{x} = \sum_{i=1}^k \lambda_i \vmap{\mathbf{q}_i}\ \right|
      \ \begin{array}{l}
      \lambda_i \geq 0\ \forall i \in 1..k, \\
      x_p \geq 0\ \forall p \in 1..7n
      \end{array}\right\}. \]
    Furthermore, because each $\vmap{\mathbf{q}_i}$ is admissible
    the constraints $x_p \geq 0$ are redundant, and so
    $\ainv{0}$ is merely the set of all non-negative linear
    combinations of $\vmap{\mathbf{q}_1},\ldots,\vmap{\mathbf{q}_k}$.
    Because truncation is linear, $\tau_0(\ainv{0})$ is likewise the set
    of all non-negative linear combinations of
    $\tau_0(\vmap{\mathbf{q}_1}),\ldots,\tau_0(\vmap{\mathbf{q}_k})$.
    Noting that neither $\tau_0$ nor $\vmapsymbol$ affects quadrilateral
    coordinates, we can make the following observations:
    \begin{itemize}
        \item Suppose that some $\tau_0(\vmap{\mathbf{q}_i}) \in \R^{7n}$
        could be expressed as a linear combination of the others.
        Restricting our attention to quadrilateral
        coordinates\footnote{More precisely, applying the
        linear map $\qmapsymbol$.  Note that
        $\qmapsymbol \circ \tau_0 \circ \vmapsymbol = \iota$, the
        identity map.}
        would therefore give
        some $\mathbf{q}_i \in \R^{3n}$ as a linear combination of the
        others, in violation of Lemma~\ref{l-vertexsplit}.

        \item Since the vectors $\tau_0(\vmap{\mathbf{q}_i})$ are all
        non-zero vectors with non-negative coordinates, they all lie to the
        same side of the hyperplane $\sum x_p = 0$.
    \end{itemize}

    It follows from Lemma~\ref{l-basis-construct} that $\tau_0(\ainv{0})$
    is a pointed polyhedral cone with $\tau_0(L_0)$ as its basis.  Moreover,
    since each $\mathbf{q}_i$ satisfies the quadrilateral constraints
    we have $\adm{\tau_0(L_0)}=\tau_0(L_0)$, and so by Lemma~\ref{l-adm-basis}
    $\tau_0(L_0)$ is a basis for $\adm{\tau_0(\ainv{0})}$ also.
    Finally, it is clear from construction that every vector in $L_0$ is
    admissible.

    \proofpart{Part~\pref{en-proof-finish}}

    We now jump straight to part~\pref{en-proof-finish}, where
    we can ignore truncations entirely because $\tau_m$ is the identity map.
    We assume therefore that $L_m$ is a basis for $\adm{\ainv{m}}$,
    and our task is to prove from this that the projective images of
    the vectors in $L_m$ together form the standard solution set for $\tri$.

    The key observation here is that the semi-admissible part
    $\adm{\ainv{m}}$ is simply $\R^{7n}_a$, the set of all admissible
    vectors in standard coordinates.  To see this:
    \begin{itemize}
        \item Every vector in $\adm{\ainv{m}}$ has non-negative
        coordinates by definition of $\ainv{m}$, and satisfies the
        quadrilateral constraints by definition of $\alpha$.
        Moreover, since every $\vmap{\mathbf{q}_i}$ and $\ell(V_j)$
        satisfies the standard matching equations, so does every vector in
        $\adm{\ainv{m}}$.  Thus $\adm{\ainv{m}} \subseteq \R^{7n}_a$.

        \item Let $\mathbf{x} \in \R^{7n}_a$.
        By Definition~\ref{d-soln-space}, the quadrilateral projection
        $\qmap{\mathbf{x}}$ can be expressed as a
        non-negative combination of vertices of the quadrilateral projective
        solution space $\quadproj$.  More specifically,
        $\qmap{\mathbf{x}}$ can be expressed as a non-negative
        combination of {\em admissible} vertices of $\quadproj$, since
        otherwise $\qmap{\mathbf{x}}$
        would not satisfy the quadrilateral constraints.
        Therefore $\qmap{\mathbf{x}} = \sum \lambda_i \mathbf{q}_i$
        for some $\lambda_i \geq 0$.

        By chasing maps around the commutative diagram in
        Lemma~\ref{l-mappings} and recalling that $\qmapsymbol$ is
        linear and $\vmapsymbol$ preserves scalar multiplication,
        we subsequently derive the equation $\kappa_v(\mathbf{x}) =
        \kappa_v\left(\sum \lambda_i \vmap{\mathbf{q}_i}\right)$.
        That is, $\mathbf{x}$ is a non-negative linear combination of
        $\vmap{\mathbf{q}_i}$ plus some arbitrary linear combination of vertex
        links.  Hence $\mathbf{x} \in \adm{\ainv{m}}$, and we have
        $\R^{7n}_a \subseteq \adm{\ainv{m}}$.
    \end{itemize}

    From here part~\pref{en-proof-finish} is straightforward.  We know
    that $L_m$ is a basis for $\R^{7n}_a=\adm{\ainv{m}}$, and
    from Example~\ref{ex-basis-soln} we know that the standard solution
    set is a basis for $\R^{7n}_a$ also.  Expressing $\R^{7n}_a$
    as the semi-admissible part of a pointed polyhedral cone
    (Example~\ref{ex-basis-soln} again), we can combine
    Lemmata~\ref{l-basis-unique} and~\ref{l-adm-basis} to show that
    the basis for $\adm{\ainv{m}}$ is unique up to scalar multiplication.
    It follows that once we take projective images, the list $L_m$ and
    the standard solution set are identical.

    \proofpart{Part~\pref{en-proof-ind}}

    All that remains is part~\pref{en-proof-ind}, the inductive step.
    Suppose we are constructing the list $L_r$ for some $r > 0$.
    Our outer inductive hypothesis is that
    the list $L_{r-1}$ consists only of admissible vectors and that
    $\tau_{r-1}(L_{r-1})$ is a basis for $\adm{\tau_{r-1}(\ainv{r-1})}$.
    Our task is to prove that, once the list $L_r$ is complete,
    it too consists only of admissible
    vectors with the truncation $\tau_r(L_r)$
    forming a basis for $\adm{\tau_r(\ainv{r})}$.

    To show this, we must dig into the construction of the list $L_r$
    and perform a new {\em inner} induction over the constructive loop in
    step~\enref{en-convert-vtx}(c) of the algorithm.  Suppose the list
    $L_{r-1}$ consists of the vectors $\mathbf{a}_1,\ldots,\mathbf{a}_t$.
    For every set of coordinate positions $P$, we define a new space
    \begin{equation}
    \label{eqn-b-defn}
      \binv{P}\ =\ \left\{
      \mathbf{x} = \sum_{i=1}^t \lambda_i \kappa_v^{(r)}(\mathbf{a}_i)
      - \mu_r \ell(V_r)\ \left|\ \begin{array}{l}
      \lambda_i \geq 0\ \forall i \in 1..t, \\
      \mu_r \geq 0, \\
      x_p \geq 0\ \forall p \in P
      \end{array}\right.\right\}.
    \end{equation}

    Essentially, $\binv{P}$ is constructed by taking non-negative linear
    combinations of (i)~the $r$th partial canonical parts of vectors
    in $L_{r-1}$, and (ii)~the negative vertex link $-\ell(V_r)$.
    Note that we relax our insistence on
    non-negative coordinates---vectors in $\binv{P}$ may include negative
    coordinates, as long as these only occur at coordinate positions outside
    the set $P$.

    Our inner inductive claim is the following.  It should be read as
    a loop invariant that applies before and after each position
    $p \in T_r$ is processed in step~\enref{en-convert-vtx}(c) of
    Algorithm~\ref{a-quadtostd}.

    \displayclaim[\claimb]{
        Every vector $\mathbf{x}$ in the partially-constructed list $L_r$
        satisfies both the standard matching equations and the quadrilateral
        constraints, and has at least one coordinate position $p \in T_r$
        for which $x_p \leq 0$.  Furthermore, the truncation
        $\tau_r(\binv{C})$ is a pointed polyhedral cone,
        and the truncation $\tau_r(L_r)$ forms a basis for the
        semi-admissible part $\adm{\tau_r(\binv{C})}$.}

    Note that both the list $L_r$ and the set $C$ continue to change as
    step~\enref{en-convert-vtx}(c) iterates through each position
    $p \in T_r$.  Our claim is that they both change together in a way
    that maintains the truth of {\claimb}.

    The remainder of this proof proceeds according to the following plan.
    Once again, the context for this plan is that we are currently
    constructing the list $L_r$ in the $r$th iteration of
    step~\enref{en-convert-vtx} of the algorithm.

    \displayroadmap{\textbf{Road map for parts
        \pref{en-proof-dd-start}--\pref{en-proof-dd-finish}:}
        \begin{enumerate}[I.]
            \setcounter{enumi}{3}
            \item \label{en-proof-dd-start}
            Show that {\claimb} is true when we first reach
            step~\enref{en-convert-vtx}(c);
            \item \label{en-proof-dd-ind}
            Show that, when processing some $p \in T_r$ in
            step~\enref{en-convert-vtx}(c),
            if {\claimb} is true before running
            step~\enref{en-convert-vtx}(c)(i)
            then {\claimb} is still true after running
            step~\enref{en-convert-vtx}(c)(iv);
            \item \label{en-proof-dd-finish}
            Show that, if {\claimb} is true after the loop in
            step~\enref{en-convert-vtx}(c) finishes,
            then {\claima} is true at the end of step~\enref{en-convert-vtx}(d).
        \end{enumerate}
    }
    In other words, parts~\pref{en-proof-dd-start} and~\pref{en-proof-dd-ind}
    constitute an inner induction to establish the correctness of the
    invariant {\claimb} throughout the construction of the list $L_r$.
    Part~\pref{en-proof-dd-finish} then uses this invariant to prove
    the outer inductive claim {\claima}, concluding
    part~\pref{en-proof-ind} and the proof of Algorithm~\ref{a-quadtostd}.

    Throughout parts~\pref{en-proof-dd-start}--\pref{en-proof-dd-finish}
    we continue to assume the outer inductive hypothesis; that is,
    that $L_{r-1}$ consists of the admissible
    vectors $\mathbf{a}_1,\ldots,\mathbf{a}_t$, and that the truncation
    $\tau_{r-1}(L_{r-1})$ forms a basis for $\adm{\tau_{r-1}(\ainv{r-1})}$.

    \proofpart{Part~\pref{en-proof-dd-start}}

    We begin our inner induction with part~\pref{en-proof-dd-start},
    at the point where we
    first reach step~\enref{en-convert-vtx}(c).  At this point in the
    algorithm, the relevant variables take the following values:
    \begin{itemize}
        \item $L_r$ consists of $\canp{r}{\mathbf{a}_1}$, \ldots,
        $\canp{r}{\mathbf{a}_t}$ and the negative vertex link $-\ell(V_r)$;
        \item $C$ consists of all quadrilateral coordinate positions,
        as well as the triangular coordinate positions in sets
        $T_1,\ldots,T_{r-1}$.
    \end{itemize}
    Our task is to show that the claim {\claimb} holds true for these values
    of $L_r$ and $C$.

    It is clear from construction that every $\mathbf{x} \in L_r$ has
    at least one coordinate position $p \in T_r$ for which $x_p \leq 0$.
    Moreover, since the outer inductive hypothesis shows that every
    $\mathbf{a}_i$ is admissible, we can see that (i)~every
    $\mathbf{x} \in L_r$ satisfies both the standard matching equations
    and the quadrilateral constraints, and that (ii)~the only
    coordinates in any $\mathbf{x} \in L_r$ that {\em might} be negative are
    those in positions $p \in T_r$.  Noting that $T_r \cap C =
    \emptyset$, the constraint $x_p \geq 0\ \forall p \in C$ in equation
    (\ref{eqn-b-defn}) is therefore redundant in this case, and we simply have
    \begin{equation}
      \label{eqn-b-first}
      \binv{C}\ =\ \left\{\left.
      \mathbf{x} = \sum_{i=1}^t \lambda_i \kappa_v^{(r)}(\mathbf{a}_i)
      - \mu_r \ell(V_r)\ \right|
      \ \begin{array}{l}
      \lambda_i \geq 0\ \forall i \in 1..t, \\
      \mu_r \geq 0
      \end{array}\right\}.
    \end{equation}
    It remains to show that $\tau_r(\binv{C})$ is a pointed polyhedral cone,
    and that $\tau_r(L_r)$ forms a basis for $\adm{\tau_r(\binv{C})}$.

    From (\ref{eqn-b-first})
    it is clear that $\binv{C}$ is the set of all non-negative
    linear combinations of vectors in $L_r$, and thus $\tau_r(\binv{C})$
    is the set of all non-negative linear combinations of vectors in
    $\tau_r(L_r)$.  We now focus on establishing the conditions
    of Lemma~\ref{l-basis-construct} for the list $\tau_r(L_r)$.
    \begin{enumerate}[(i)]
        \item Suppose that some vector in $\tau_r(L_r)$ can be written as a
        non-negative linear combination of the other vectors in
        $\tau_r(L_r)$.  Taking the linear map $\tau_{r-1}$ and recalling
        that $\tau_{r-1} \circ \tau_r = \tau_{r-1}$,
        it follows that the corresponding vector in
        $\tau_{r-1}(L_r)$ can be written as
        the same non-negative linear combination of the other vectors in
        $\tau_{r-1}(L_r)$.

        For each $\kappa_v^{(r)}(\mathbf{a}_i) \in L_r$ we have
        $\tau_{r-1}(\kappa_v^{(r)}(\mathbf{a}_i)) = \tau_{r-1}(\mathbf{a}_i)$,
        and for $-\ell(V_r) \in L_r$ we have
        $\tau_{r-1}(-\ell(V_r))=\mathbf{0}$.
        Thus $\tau_{r-1}(L_r)$ consists of the basis $\tau_{r-1}(L_{r-1})$
        combined
        with the zero vector, and so the only possible non-negative linear
        combination in $\tau_{r-1}(L_r)$ is the trivial combination
        $\tau_{r-1}(-\ell(V_r)) = \mathbf{0}$.  It follows that
        our original non-negative linear combination in $\tau_r(L_r)$ must
        have been $\tau_r(-\ell(V_r)) = \mathbf{0}$, a contradiction.

        \item We aim now to construct a hyperplane $H \subset \R^{7n}$
        for which every vector in $\tau_r(L_r)$ lies strictly to the same
        side of $H$.  To do this, we define the temporary vector
        $\mathbf{u}=\tau_{r-1}(\mathbf{1})$.  That is, $\mathbf{u}$ contains
        1 in all quadrilateral coordinate positions as well as the
        triangle positions $p \in T_1 \cup \ldots \cup T_{r-1}$,
        and contains 0 in the remaining triangle positions
        $p \in T_r \cup \ldots \cup T_m$.  Recall also that the vertex
        link $\ell(V_r)$ contains 1 in all triangle positions $p \in T_r$,
        and contains 0 in all other triangle and quadrilateral positions.

        Define the constants
        \[ g=\min_{i=1}^t \left\{
          \mathbf{u}\cdot\tau_r(\kappa_v^{(r)}(\mathbf{a}_i))\right\}
        \quad\text{and}\quad
        h=\max_{i=1}^t \left\{
          \ell(V_r)\cdot\tau_r(\kappa_v^{(r)}(\mathbf{a}_i))\right\}.\]
        Since $\tau_{r-1}(L_{r-1})$ is a basis of non-negative vectors
        and $\mathbf{u}\cdot\tau_r(\kappa_v^{(r)}(\mathbf{a}_i)) =
        \mathbf{u}\cdot\tau_{r-1}(\mathbf{a}_i)$, it is clear that $g > 0$.
        Furthermore, from the definition of $\kappa_v^{(r)}$ and the
        fact that $\ell(V_r)\cdot\tau_r(\kappa_v^{(r)}(\mathbf{a}_i)) =
        \ell(V_r)\cdot\kappa_v^{(r)}(\mathbf{a}_i)$ it is clear that $h \geq 0$.

        Let $H$ be the hyperplane
        $\{ \mathbf{x}\in\R^{7n}\,|\,
        (h+1)\mathbf{u} \cdot \mathbf{x} = g\,\ell(V_r) \cdot \mathbf{x}\}$.
        We show now that every element of $\tau_r(L_r)$ lies strictly to the
        same side of $H$.
        By definition of $g$ and $h$, if
        $\mathbf{x}=\tau_r(\kappa_v^{(r)}(\mathbf{a}_i)) \in \tau_r(L_r)$
        then we have $(h+1)\mathbf{u} \cdot \mathbf{x} \geq (h+1)g >
        g\,h \geq g\,\ell(V_r) \cdot \mathbf{x}$.  Finally, if
        $\mathbf{x}=\tau_r(-\ell(V_r)) \in \tau_r(L_r)$ then this simplifies
        to $\mathbf{x}=-\ell(V_r)$, and so
        $(h+1)\mathbf{u} \cdot \mathbf{x} = 0 > g\,\ell(V_r) \cdot
        \mathbf{x}$.  Therefore $H$ is the hyperplane that we require.
    \end{enumerate}

    It follows from Lemma~\ref{l-basis-construct} that $\tau_r(\binv{C})$
    is a pointed polyhedral cone with $\tau_r(L_r)$ as its basis.
    Finally, since every $\mathbf{a}_i$ is admissible it is clear that
    every vector of $\tau_r(L_r)$ satisfies the quadrilateral constraints;
    thus $\adm{\tau_r(L_r)} = \tau_r(L_r)$, and Lemma~\ref{l-adm-basis}
    shows that $\tau_r(L_r)$ is a basis for $\adm{\tau_r(\binv{C})}$ also.

    \proofpart{Part~\pref{en-proof-dd-ind}}

    We come now to part~\pref{en-proof-dd-ind}, the main inductive step
    for the inner induction.  Here we assume that {\claimb} holds
    before running step~\enref{en-convert-vtx}(c)(i); our task is to show
    that {\claimb} still holds after running step~\enref{en-convert-vtx}(c)(iv).

    Throughout this part, we assume that we are building the list $L_r$,
    and that we are currently processing some coordinate position $p \in T_r$.
    We use the following notation:
    \begin{itemize}
        \item $L_r$ and $C$ denote the {\em initial} states of these
        variables, before step~\enref{en-convert-vtx}(c)(i).

        \item $S_0$, $S_+$ and $S_-$ are as defined in
        Algorithm~\ref{a-quadtostd}; that is,
        $S_0 = \{ \mathbf{x} \in L_r\,|\,x_p = 0\}$,
        $S_+ = \{ \mathbf{x} \in L_r\,|\,x_p > 0\}$ and
        $S_- = \{ \mathbf{x} \in L_r\,|\,x_p < 0\}$.

        \item $L'$ denotes the final state of the list after
        step~\enref{en-convert-vtx}(c)(iv).  In other words,
    \end{itemize}
    \begin{equation}
    \label{eqn-dd-ldash}
    L' = S_0 \cup S_+ \cup \left\{
        \frac{u_p\mathbf{w}-w_p\mathbf{u}}{u_p-w_p}\,\left|\,
        \begin{array}{l}
        \mathbf{u} \in S_+\ \text{and}\ \mathbf{w} \in S_-, \\
        \text{$\mathbf{u},\mathbf{w}$ together satisfy the
          quad.~constraints,} \\
        \nexists \mathbf{z} \in L_r\backslash\{\mathbf{u},\mathbf{w}\}
          \ \text{for which} \\
        \quad \left(i \in C \cup \{p\}\ \text{and}\ u_i=w_i=0\right)
          \Rightarrow z_i=0
        \end{array}
    \right.\right\}.
    \end{equation}
    In addition, we note that the final state of the set $C$ is simply
    $C \cup \{p\}$.
    We can therefore assume claim {\claimb} exactly as written, and our
    task is to prove the following:
    \begin{enumerate}[(a)]
        \item \label{en-ind-dd-vecprop}
        Every $\mathbf{x} \in L'$ satisfies both the standard
        matching equations and the quadrilateral constraints, and has at
        least one coordinate position $p' \in T_r$ for which $x_{p'} \leq 0$;
        \item \label{en-ind-dd-cone}
        The truncation $\tau_r(\binv{C \cup \{p\}})$ is a
        pointed polyhedral cone;
        \item \label{en-ind-dd-basis}
        The truncation $\tau_r(L')$ forms a basis for
        $\adm{\tau_r(\binv{C \cup \{p\}})}$.
    \end{enumerate}
    Claim~(\enref{en-ind-dd-vecprop}) is straightforward; these properties
    are already known to be true for all vectors in $S_0,S_+,S_- \subseteq L_r$,
    and it is clear by construction that they also hold for vectors new
    to $L'$.  In particular, $x_p=0$ for each new vector
    $\mathbf{x} = (u_p\mathbf{w}-w_p\mathbf{u})/(u_p-w_p)$.
    Claim~(\enref{en-ind-dd-cone}) is also straightforward, since
    equation~(\ref{eqn-b-defn}) shows that $\tau_r(\binv{C \cup \{p\}})$
    is the intersection of the pointed polyhedral cone $\tau_r(\binv{C})$
    with the half-space $x_p \geq 0$.  We therefore focus our efforts
    on proving~(\enref{en-ind-dd-basis}), that is, that
    $\tau_r(L')$ forms a basis for $\adm{\tau_r(\binv{C \cup \{p\}})}$.

    We know from {\claimb} that $L_r$ forms a basis for
    $\adm{\tau_r(\binv{C})}$, where $\tau_r(\binv{C})$ is a pointed polyhedral
    cone.  From Lemma~\ref{l-adm-basis}, there is a basis $M$ for
    $\tau_r(\binv{C})$ for which $L_r = \adm{M}$.  As noted earlier,
    the final polyhedral cone $\tau_r(\binv{C \cup \{p\}})$ is simply
    $\tau_r(\binv{C})$ intersected with the half-space $x_p \geq 0$;
    our plan is to use this fact to convert $M$ into a basis for
    $\tau_r(\binv{C \cup \{p\}})$ and then a basis for
    $\adm{\tau_r(\binv{C \cup \{p\}})}$, which we will see
    is simply the final list $L'$.

    To convert $M$ into a basis for $\tau_r(\binv{C \cup \{p\}})$,
    we call upon the regular double description method.
    Just as $M$ is a superset of $L_r$, we define the supersets
    $M_0 = \{\mathbf{m} \in M\,|\,m_p = 0\} \supseteq S_0$,
    $M_+ = \{\mathbf{m} \in M\,|\,m_p > 0\} \supseteq S_+$ and
    $M_- = \{\mathbf{m} \in M\,|\,m_p < 0\} \supseteq S_-$.
    Lemma~\ref{l-dd-core} then shows that the following
    is a basis for $\tau_r(\binv{C \cup \{p\}}) =
    \tau_r(\binv{C}) \cap \{\mathbf{x}\,|\,x_p \geq 0\}$:
    \begin{equation*}
    M' = M_0 \cup M_+ \cup \left\{
        \frac{u_p\mathbf{w} - w_p\mathbf{u}}{u_p - w_p}
        \,\left|\,
        \begin{array}{l}
        \mathbf{u} \in M_+\ \text{and}\ \mathbf{w} \in M_-, \\
        \text{$\mathbf{u},\mathbf{w}$ are adjacent basis vectors
          in $\tau_r(\binv{C})$}
        \end{array}
    \right.\right\}.
    \end{equation*}
    Using Lemma~\ref{l-adm-basis} and the observation that
    $u_p > 0 > w_p$, a corresponding basis for the
    semi-admissible part $\adm{\tau_r(\binv{C \cup \{p\}})}$ is
    \begin{equation*}
    \adm{M'} = S_0 \cup S_+ \cup \left\{
        \frac{u_p\mathbf{w} - w_p\mathbf{u}}{u_p - w_p}
        \,\left|\,
        \begin{array}{l}
        \mathbf{u} \in S_+\ \text{and}\ \mathbf{w} \in S_-, \\
        \text{$\mathbf{u},\mathbf{w}$ together satisfy the
          quad.~constraints,} \\
        \text{$\mathbf{u},\mathbf{w}$ are adjacent basis vectors
          in $\tau_r(\binv{C})$}
        \end{array}\right.\right\}.
    \end{equation*}
    Consider the following claim, which we will prove shortly.
    \displayclaim[\claimc]{Suppose $\mathbf{u}$ and $\mathbf{w}$ are
        basis vectors for $\tau_r(\binv{C})$ that together satisfy the
        quadrilateral constraints.  Then $\mathbf{u}$ and $\mathbf{w}$
        are adjacent if and only if there is no
        $\mathbf{z} \in M\backslash\{\mathbf{u},\mathbf{w}\}$ for which,
        whenever $i \in C$ and $u_i=w_i=0$, we must have $z_i=0$.}
    \noindent
    If this is true, then our basis for $\adm{\tau_r(\binv{C \cup \{p\}})}$
    can be rewritten as
    \begin{equation*}
    \adm{M'} = S_0 \cup S_+ \cup \left\{
        \frac{u_p\mathbf{w} - w_p\mathbf{u}}{u_p - w_p}
        \,\left|\,
        \begin{array}{l}
        \mathbf{u} \in S_+\ \text{and}\ \mathbf{w} \in S_-, \\
        \text{$\mathbf{u},\mathbf{w}$ together satisfy the
          quad.~constraints,} \\
        \nexists \mathbf{z} \in M\backslash\{\mathbf{u},\mathbf{w}\}
          \ \text{for which} \\
        \quad \left(i \in C\ \text{and}\ u_i=w_i=0\right) \Rightarrow z_i=0
        \end{array}\right.\right\}.
    \end{equation*}
    Because $C$ contains all quadrilateral positions, we can change
    $\mathbf{z} \in M\backslash\{\mathbf{u},\mathbf{w}\}$ in the final
    condition above to $\mathbf{z} \in L_r\backslash\{\mathbf{u},\mathbf{w}\}$.
    Furthermore, because $u_p,w_p \neq 0$ we can change $i \in C$ in this same
    condition to $i \in C \cup \{p\}$.  The equation then becomes identical
    to (\ref{eqn-dd-ldash}), and we see that our basis $\adm{M'}$ is
    indeed the final list $L'$.

    The only thing now remaining for part~\pref{en-proof-dd-ind} is to
    prove the claim {\claimc}.  We do this using
    Lemma~\ref{l-basis-adj-lincomb}.

    Suppose that $\mathbf{u}$ and $\mathbf{w}$ are {\em not} adjacent
    basis vectors in $\tau_r(\binv{C})$.  By Lemma~\ref{l-basis-adj-lincomb}
    there is some $\mathbf{x} \in \tau_r(\binv{C})$ and some coefficients
    $\mu,\eta,\lambda_j \geq 0$ for which
    $\mathbf{x} = \mu \mathbf{u} + \eta \mathbf{w} = \sum \lambda_j
    \mathbf{b}_j$, where each $\mathbf{b}_j$ is a basis vector for
    $\tau_r(\binv{C})$ and where $\lambda_j > 0$ for some
    $\mathbf{b}_j \neq \mathbf{u},\mathbf{w}$.
    Because the $i$th coordinate of every basis vector is non-negative
    for every $i \in C$, it follows that whenever
    $u_i=w_i=0$ for $i \in C$ we must have $(\mathbf{b}_j)_i=0$ also.
    Therefore $\mathbf{b}_j$ satisfies the conditions for $\mathbf{z}$
    as specified in {\claimc}.

    Suppose now that $\mathbf{u}$ and $\mathbf{w}$ {\em are} adjacent
    basis vectors in $\tau_r(\binv{C})$ that together satisfy the
    quadrilateral constraints, and that $\mathbf{z}$ is some
    different basis vector for which, whenever $i \in C$ and $u_i=w_i=0$,
    we have $z_i=0$ also.  We show that this leads to a contradiction.

    Let $\mathbf{x}=\mathbf{u}+\xi\mathbf{w}$ where $\xi>0$ is
    chosen so that $x_p < 0$, and let
    $\mathbf{y}=\mathbf{x}-\epsilon\mathbf{z}$ for some small $\epsilon>0$.
    If $\mathbf{y} \in \tau_r(\binv{C})$ then we can express
    $\mathbf{y}=\sum \lambda_i \mathbf{b}_i$, where each $\lambda_i \geq 0$
    and each $\mathbf{b}_i$ is a basis vector for $\tau_r(\binv{C})$.
    This gives us
    $\mathbf{x} = \mathbf{u}+\xi\mathbf{w} =
    \epsilon\mathbf{z}+\sum \lambda_i \mathbf{b}_i$, whereupon
    Lemma~\ref{l-basis-adj-lincomb} shows that
    $\mathbf{u}$ and $\mathbf{w}$ cannot be adjacent, giving us the
    contradiction that we seek.

    It remains to prove that $\mathbf{y} \in \tau_r(\binv{C})$.
    The condition on $\mathbf{z}$ ensures that for
    sufficiently small $\epsilon>0$ we have $y_i \geq 0$ for all $i \in C$,
    so all we need to show is that $\mathbf{y}$ can be expressed as a
    linear combination
    $\mathbf{y} = \sum_i\lambda_i\kappa_v^{(r)}(\mathbf{a}_i)-\mu_r\ell(V_r)$
    for $\lambda_i,\mu_r \geq 0$.

    From (\ref{eqn-b-defn}) all vectors in $\binv{C}$ satisfy the
    standard matching equations.  Since $\mathbf{y}$ is a linear combination
    of vectors in $\tau_r(\binv{C})$, it follows that
    $\mathbf{y}=\tau_r(\mathbf{y}')$ for some $\mathbf{y}' \in \R^{7n}$
    that also satisfies the standard matching equations.  Because
    $y_i \geq 0$ for all $i \in C$ we see that both $\mathbf{y}$ and
    $\mathbf{y}'$ satisfy the quadrilateral constraints, and that
    $\mathbf{y}' + \sum_{i=r}^m \zeta_i \ell(V_i)$ is a
    non-negative vector for some coefficients $\zeta_r,\ldots,\zeta_m \in \R$.
    Therefore $\mathbf{y}' + \sum_{i=r}^m \zeta_i \ell(V_i)$ is admissible.

    It follows that $\qmap{\mathbf{y}' + \sum_{i=r}^m \zeta_i \ell(V_i)}
    =\qmap{\mathbf{y}'} \in \R^{3n}$ can be expressed as a non-negative
    linear combination of vectors in the quadrilateral solution set,
    and so from equation~(\ref{eqn-a-defn}) we see that
    $\tau_{r-1}(\mathbf{y}' + \sum_{i=r}^m \zeta_i \ell(V_i)) =
    \tau_{r-1}(\mathbf{y}') \in \tau_{r-1}(\ainv{r-1})$.
    Using the quadrilateral constraints for $\mathbf{y}'$ we then obtain
    $\tau_{r-1}(\mathbf{y}') \in \adm{\tau_{r-1}(\ainv{r-1})}$, and so
    $\tau_{r-1}(\mathbf{y}') =
    \sum_{i=1}^t \lambda_i\tau_{r-1}(\mathbf{a}_i)$ for some
    $\lambda_1,\ldots,\lambda_t \geq 0$.

    Because $\mathbf{y'}$ is admissible, Lemma~\ref{l-trunc-error}
    shows that the only error we can
    introduce by replacing $\tau_{r-1}$ with $\tau_r$ is a multiple
    of the vertex link $\ell(V_r)$.  Therefore
    $\mathbf{y} = \tau_r(\mathbf{y}') =
    \sum_{i=1}^t \lambda_i\tau_r(\mathbf{a}_i) + \mu \ell(V_r)$
    for some coefficient $\mu \in \R$.  Since the partial canonical part
    $\kappa_v^{(r)}$ only adds or subtracts multiples of $\ell(V_r)$,
    we can rewrite this as
    $\mathbf{y} =
    \sum_{i=1}^t\lambda_i\kappa_v^{(r)}(\tau_r(\mathbf{a}_i))+\mu'\ell(V_r)$.
    Finally, because we chose $\mathbf{x}$ and $\mathbf{y}$ to satisfy
    $y_p \leq x_p < 0$ we must have $\mu'<0$, and equation~(\ref{eqn-b-defn})
    shows that $\mathbf{y} \in \tau_r(\binv{C})$ as required.

    \proofpart{Part~\pref{en-proof-dd-finish}}

    Moving to the final part~\pref{en-proof-dd-finish}, we now assume that
    {\claimb} holds at the end of step~\enref{en-convert-vtx}(c) of the
    algorithm; our task then is to prove that {\claima} holds at the end of
    step~\enref{en-convert-vtx}(d).

    At this stage of the algorithm, the set $C$ contains all positions
    $p \in T_r$ (amongst others).  Consider any $\mathbf{x} \in \binv{C}$.
    We know that $\mathbf{x}$ can be expressed as a non-negative vector
    minus $\mu_r \ell(V_r)$, but we also know that
    $x_p \geq 0$ for all $p \in C \supseteq T_r$.  It follows that every
    $\mathbf{x} \in \binv{C}$ is a non-negative vector, and so in this
    case we can write $\binv{C}$ as
    \begin{equation}
      \label{eqn-b-last}
      \binv{C}\ =\ \left\{
      \mathbf{x} = \sum_{i=1}^t \lambda_i \kappa_v^{(r)}(\mathbf{a}_i)
      - \mu_r \ell(V_r)\ \left|\ \begin{array}{l}
      \lambda_i \geq 0\ \forall i \in 1..t, \\
      \mu_r \geq 0, \\
      x_p \geq 0\ \forall p \in 1..7n
      \end{array}\right.\right\}.
    \end{equation}
    That is, we can replace the specific condition
    $x_p \geq 0\ \forall p \in C$ with the more general condition
    $x_p \geq 0\ \forall p \in 1..7n$.

    We pick off the easy part of {\claima} first.  From {\claimb} we know
    that after step~\enref{en-convert-vtx}(c)
    every $\mathbf{x} \in L_r$ satisfies both the standard matching
    equations and the quadrilateral constraints, and from (\ref{eqn-b-last})
    every $\mathbf{x} \in L_r$ is a non-negative vector also.
    Thus $L_r$ consists only of admissible vectors, and inserting the
    vertex link in step~\enref{en-convert-vtx}(d) does not change this fact.

    It remains to prove that $\tau_r(L_r)$ forms a basis for
    $\adm{\tau_r(\ainv{r})}$.  We do this directly through
    Definition~\ref{d-basis}.
    \begin{enumerate}[(i)]
        \item At the end of step~\enref{en-convert-vtx}(c) of the algorithm,
        we know from {\claimb} that $\tau_r(L_r)$ forms a basis for
        $\adm{\tau_r(\binv{C})}$.  It follows that
        $\tau_r(L_r) \subseteq \adm{\tau_r(\binv{C})}$, and that
        every $\mathbf{x} \in \adm{\tau_r(\binv{C})}$ can be expressed as a
        non-negative linear combination of vectors in $\tau_r(L_r)$.
        We aim to show the same for every
        $\mathbf{x} \in \adm{\tau_r(\ainv{r})}$
        at the end of step~\enref{en-convert-vtx}(d).

        It can be seen from the definition of $\ainv{r}$ that
        \begin{align*}
        \ainv{r}\ &=\ \{ \mathbf{x}=\mathbf{a} + \mu \ell(V_r)\,|\,
           \mathbf{a} \in \ainv{r-1}\ \text{and}
           \ x_p \geq 0\ \forall p \in 1..7n \},\ \text{and hence} \\
        \adm{\tau_r(\ainv{r})}\ &=\ \{ \mathbf{x}=\mathbf{a} +
           \mu \ell(V_r)\,|\,\mathbf{a} \in \adm{\tau_r(\ainv{r-1})}
           \ \text{and}\ x_p \geq 0\ \forall p \in 1..7n \}.
        \end{align*}
        We now call upon the outer inductive hypothesis; in particular,
        the fact that $\tau_{r-1}(L_{r-1})$ is a
        basis for $\adm{\tau_{r-1}(\ainv{r-1})}$.  Combining this with
        Lemma~\ref{l-trunc-error} to replace $\tau_{r-1}$ with $\tau_r$,
        our equation becomes
        \begin{equation*}
          \adm{\tau_r(\ainv{r})}\ =\ \alpha\left(\left\{
          \mathbf{x} = \sum_{i=1}^t \lambda_i \tau_r(\mathbf{a}_i)
          + \mu \ell(V_r)\ \left|\ \begin{array}{l}
          \lambda_i \geq 0\ \forall i \in 1..t, \\
          x_p \geq 0\ \forall p \in 1..7n
          \end{array}\right.\right\}\right).
        \end{equation*}
        Finally, using equation~(\ref{eqn-b-last}) and the fact that
        $\kappa_v^{(r)}$ only ever adds or subtracts copies of $\ell(V_r)$,
        we obtain
        \[ \adm{\tau_r(\ainv{r})} =
           \{\mathbf{b} + \mu \ell(V_r)\,|
           \,\mathbf{b} \in \adm{\tau_r(\binv{C})}\ \text{and}
           \ \mu \geq 0\}. \]

        That is, $\adm{\tau_r(\ainv{r})}$ consists of
        all non-negative linear combinations of (a)~vectors in
        $\adm{\tau_r(\binv{C})}$, and (b)~the vertex link $\ell(V_r)$.
        It follows from {\claimb} that, once we insert the vertex link into
        $L_r$ in step~\enref{en-convert-vtx}(d) of the algorithm, we know that
        $\tau_r(L_r) \subseteq \adm{\tau_r(\ainv{r})}$ and that
        every $\mathbf{x} \in \adm{\tau_r(\ainv{r})}$ can be expressed as a
        non-negative linear combination of vectors in $\tau_r(L_r)$.

        \item We now show that, after step~\enref{en-convert-vtx}(d)
        of the algorithm, no vector in $\tau_r(L_r)$ can be
        expressed as a non-negative linear combination of the others.
        Let $L_r'$ be the list $L_r$ as it was immediately after
        step~\enref{en-convert-vtx}(c) (that is, without the vertex link);
        from {\claimb} we know this property is true for $\tau_r(L_r')$.
        Denote the vectors in $L_r'$ as $\mathbf{b}_1,\ldots,\mathbf{b}_q$.

        Suppose that some vector in $\tau_r(L_r)$ can be expressed
        as a non-negative linear combination of the others.
        Because the list $\tau_r(L_r)$ contains only the basis elements
        $\tau_r(\mathbf{b}_1),\ldots,\tau_r(\mathbf{b}_q)$ and the
        vertex link $\ell(V_r)$,
        our expression must be of one of the following two types:
        \begin{itemize}
            \item $\tau_r(\mathbf{b}_i) =
            \sum_{j \neq i} \lambda_j \tau_r(\mathbf{b}_j)
            + \mu \ell(V_r)$ for $\lambda_j \geq 0$ and $\mu > 0$.
            That is, the vertex link $\ell(V_r)$ appears as a non-empty
            part of this linear combination.

            \smallskip

            Because $\mathbf{b}_i$ is a non-negative vector,
            the clause $x_p \leq 0$ in {\claimb} implies that
            $\kappa_v(\mathbf{b}_i) = \mathbf{b}_i$.
            However, because every $\mathbf{b}_j$ is also a non-negative
            vector, the presence of the vertex link on the right hand
            side above implies that
            \[ \kappa_v\left(\sum_{j \neq i} \lambda_j \tau_r(\mathbf{b}_j)
            + \mu \ell(V_r)\right) \neq
            \sum_{j \neq i} \lambda_j \tau_r(\mathbf{b}_j) + \mu \ell(V_r). \]
            That is, $\kappa_v(\mathbf{b}_i) \neq \mathbf{b}_i$, giving us
            a contradiction.

            \item $\ell(V_r) = \sum_j \lambda_j \tau_r(\mathbf{b}_j)$
            for $\lambda_j \geq 0$.  That is, the vertex link $\ell(V_r)$
            can be expressed as a non-negative linear combination of
            truncated vectors in $L_r'$.

            \smallskip

            Since all $\mathbf{b}_j$ are non-negative, every
            $\mathbf{b}_j$ that features in this linear combination
            must have all its quadrilateral coordinates equal to zero.
            Each such $\mathbf{b}_j$ is also admissible, whereupon
            Lemma~\ref{l-mappings} can be used to show it is a non-negative
            combination of vertex links.  More precisely,
            non-negativity again shows that
            each corresponding $\tau_r(\mathbf{b}_j)$ must be a
            multiple of the single vertex link $\ell(V_r)$.  However,
            this yields the expression $\tau_r(\mathbf{b}_j) = \mu \ell(V_r)$,
            which we have shown above to be impossible.
        \end{itemize}
    \end{enumerate}
    This concludes the requirements for Definition~\ref{d-basis},
    whereupon we see that $\tau_r(L_r)$ must form a basis for
    $\adm{\tau_r(\ainv{r})}$.  Indeed, this also concludes
    part~\pref{en-proof-dd-finish}, and therefore the entire proof of
    Algorithm~\ref{a-quadtostd}.
\end{proof}

\subsection{Time Complexity and the Enumeration Algorithm}
\label{s-quadtostd-consequences}

We now return to the issue of time complexity, which was raised briefly
following the statement of the quadrilateral-to-standard solution set
conversion algorithm (Algorithm~\ref{a-quadtostd}).
It has already been noted that
this conversion algorithm can grow exponentially slow in the size of the
input; it is also seen in \cite{burton08-dd} that the enumeration algorithms
for the standard and quadrilateral solution space suffer from the same problem.

We have already discussed examples where the size of the standard solution
space is exponential in $n$ (punishing the enumeration algorithm)
and also exponential in the size of the quadrilateral solution space
(punishing the conversion algorithm).  However, this is not our worst
problem.  The intermediate lists that are created by these algorithms can
potentially grow exponentially large with respect to both the
input {\em and} the output, leading to situations where both the standard and
quadrilateral solution space are very small, yet the enumeration algorithms
take a very long time to run.

The root of the problem lies in the double description method, upon which
the enumeration algorithms are built.  Using
Lemma~\ref{l-dd-core}, the double description method inductively builds
a series of lists, the last of which becomes the standard or quadrilateral
solution space.  It is well known that the double description method can
suffer from a combinatorial explosion, where the intermediate lists
can grow exponentially large before shrinking back down
to what might be a very small output set.  See
\cite{avis97-howgood-compgeom,fukuda96-doubledesc} for discussions of how this
combinatorial explosion can be tamed in general, and \cite{burton08-dd}
for techniques specific to normal surface enumeration.

Because the quadrilateral-to-standard conversion algorithm incorporates
aspects of the double description method, one should expect it to suffer
from the same problems.  However, empirical evidence suggests that it does
not---in Section~\ref{s-expt} we find that the intermediate lists in
Algorithm~\ref{a-quadtostd} appear {\em not} to explode in size (never
growing larger than $1\frac12$ times the output size), and that the total
running time for conversion appears to be negligible in comparison to
enumeration.  In light of these observations, we put forward the
following proposal.

\begin{conjecture} \label{cj-complexity}
    The time complexity of Algorithm~\ref{a-quadtostd} is at worst
    polynomial in the size of the output.  That is, the running time
    is at most a polynomial function of $n$ (the number of tetrahedra)
    and $k'$ (the size of the standard solution set).
\end{conjecture}

More specifically, it seems reasonable to believe based on experimental
evidence that the intermediate lists for Algorithm~\ref{a-quadtostd}
are at worst linear in $k'$, from which Conjecture~\ref{cj-complexity}
would follow as an immediate consequence.  A possible cause could
be the highly structured ways in which the intermediate polyhedral cones
$\ainv{r}$ and $\binv{C}$ are formed in the proof of
Algorithm~\ref{a-quadtostd}.

We finish this section with the new enumeration algorithm that was promised in
the introduction and again at the beginning of Section~\ref{s-quadtostd}.
Specifically, we use Algorithm~\ref{a-quadtostd} as a key component in a new
algorithm for enumerating the standard solution space.  As discussed in the
introduction, the enumeration problem has great practical significance
in normal surface theory but suffers from the feasibility problems
of an exponential running time.
In this context, the new algorithm below is a
significant improvement---we find in Section~\ref{s-expt} that for large
cases it runs thousands and even millions of times
faster than the current state-of-the-art.

This current state-of-the-art is described in \cite{burton08-dd};
essentially we begin with the double description method of
Motzkin et al.~\cite{motzkin53-dd}, apply the filtering techniques of
Letscher, and then incorporate a range of further improvements that exploit
special properties of the normal surface enumeration problem.
We refer to this modified double description method
as {\em direct enumeration}.

Our new enumeration algorithm combines direct enumeration with
Algorithm~\ref{a-quadtostd}, and runs as follows.

\begin{algorithm} \label{a-enumstd}
    To compute the standard solution set for the triangulation $\tri$,
    we can use the following algorithm.
    \begin{enumerate}[1.]
        \item \label{en-enumstd-quad}
        Use direct enumeration to compute the
        quadrilateral solution set for $\tri$.
        \item \label{en-enumstd-convert}
        Use Algorithm~\ref{a-quadtostd} to convert this
        quadrilateral solution set into the
        standard solution set for $\tri$.
    \end{enumerate}
\end{algorithm}

We expect this algorithm to perform well---although the direct enumeration
in quadrilateral coordinates (step~\enref{en-enumstd-quad}) remains
exponentially slow, in practice it runs many orders of magnitude faster
than a direct enumeration in standard coordinates \cite{burton08-dd}.
Following this, the quadrilateral-to-standard conversion
(step~\enref{en-enumstd-convert}) is found to run extremely quickly,
as discussed above.

All that remains is to test these claims in practice, which brings us to the
final section of this paper.

%% file: expt.tex
\section{Measuring Performance} \label{s-expt}

To conclude this paper we measure the performance of our new algorithms
through a series of practical tests.  These tests involve running both old
and new algorithms over 500 different triangulations, taking a variety of
measurements along the way.

The triangulations chosen for these tests are the first 500 orientable
triangulations from the Hodgson--Weeks closed hyperbolic census
\cite{hodgson94-closedhypcensus}; their sizes range from 9 to 25 tetrahedra.
All computations were performed on a single 2.3\,GHz AMD Opteron processor
using the software package {\regina} \cite{regina,burton04-regina}.
There are alternative implementations of normal surface enumeration available,
notably the {\fxrays} software by Culler and Dunfield \cite{fxrays};
we use {\regina} here because, with the improvements of \cite{burton08-dd},
it is found in the author's experience to have the greater efficiency in
both time and memory for large triangulations.\footnote{This observation
concerns direct enumeration (prior to this paper).
As seen in the following graphs, the new algorithms
developed in this paper are significantly more efficient again.}

Our first tests compare running times for the new enumeration algorithm in
standard coordinates (Algorithm~\ref{a-enumstd}) against the old
state-of-the-art (the modified double description method
of \cite{burton08-dd}, referred to earlier as ``direct enumeration'').
The following observations can be made:

\begin{itemize}
    \item
    Figure~\ref{fig-speed} plots new running times directly against old
    running times, with one point for each of the 500 triangulations.
    Both axes use a log scale, since running times for both algorithms
    are spread out across several orders of magnitude.
    The diagonal lines are guides to illustrate the
    magnitude of the improvements.

    \begin{figure}[htb]
    \centerline{\includegraphics[scale=0.55]{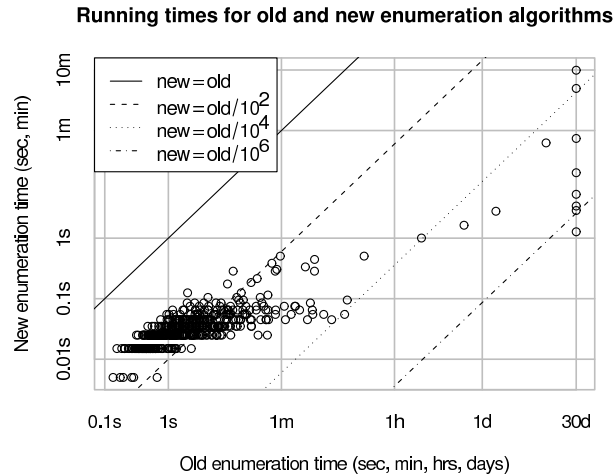}}
    \caption{Comparing the old direct enumeration
        against the new Algorithm~\ref{a-enumstd}}
    \label{fig-speed}
    \end{figure}

    It is immediately clear that the new algorithm is faster, and
    significantly so.  The weakest improvement is still over $10$
    times the speed, and the strongest is over $2\,000\,000$ times.
    Roughly speaking, the largest cases experience the
    greatest improvements (which is what we hope for).
    Some additional points worth noting:
    \begin{itemize}
        \item[--] The resolution of the timer is $0.01$ seconds.  This
        explains the long horizontal clumps in the bottom-left corner of
        the graph---here the new algorithm runs in literally the smallest
        times that can be measured.  An error factor of $0.005$ seconds
        has been added to all measurements to compensate for cases where
        the time is measured to be zero.

        \item[--] Whilst the new algorithm ran to completion for all 500
        triangulations, the old algorithm did not.  Eight cases were
        terminated after 30 days of running time; these are the eight
        points at the rightmost end of the plot.  This
        early termination underestimates the improvements due to the new
        algorithm; the real improvements might well be orders of
        magnitude larger again.
    \end{itemize}

    \item In Figure~\ref{fig-largefactor} we plot the improvement factor
    (the old running time divided by the new running time)
    against both the input size and the output size (the size of the
    quadrilateral and standard solution sets respectively).

    \begin{figure}[htb]
    \centerline{
        \includegraphics[scale=0.55]{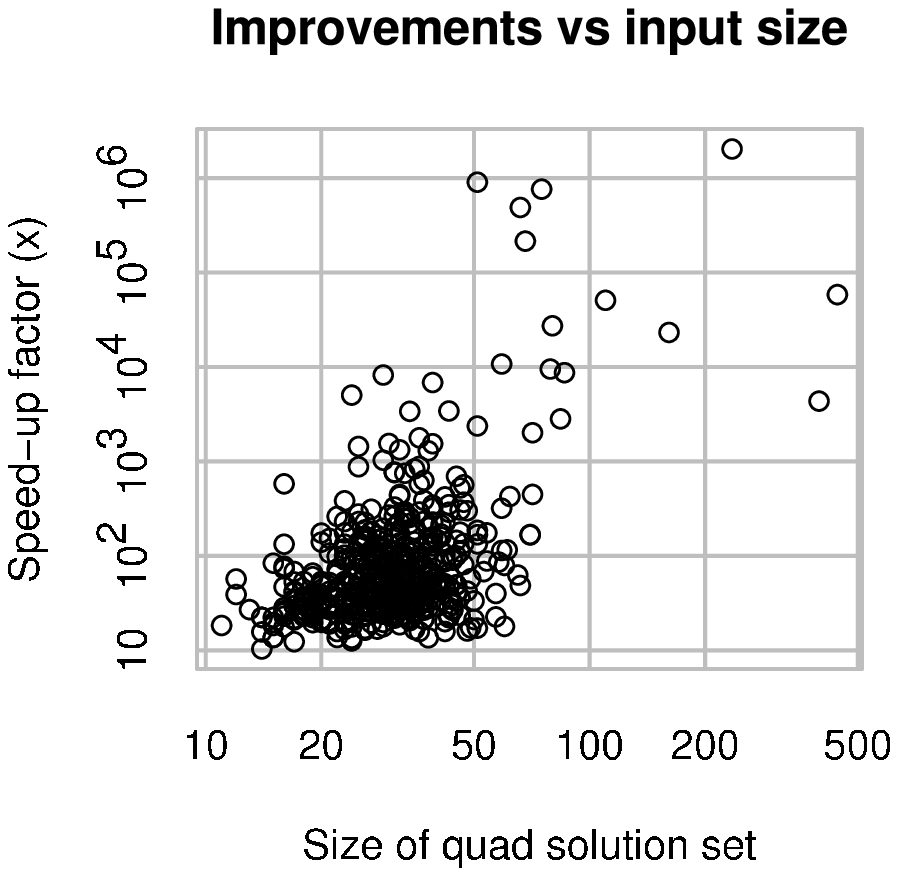}\quad
        \includegraphics[scale=0.55]{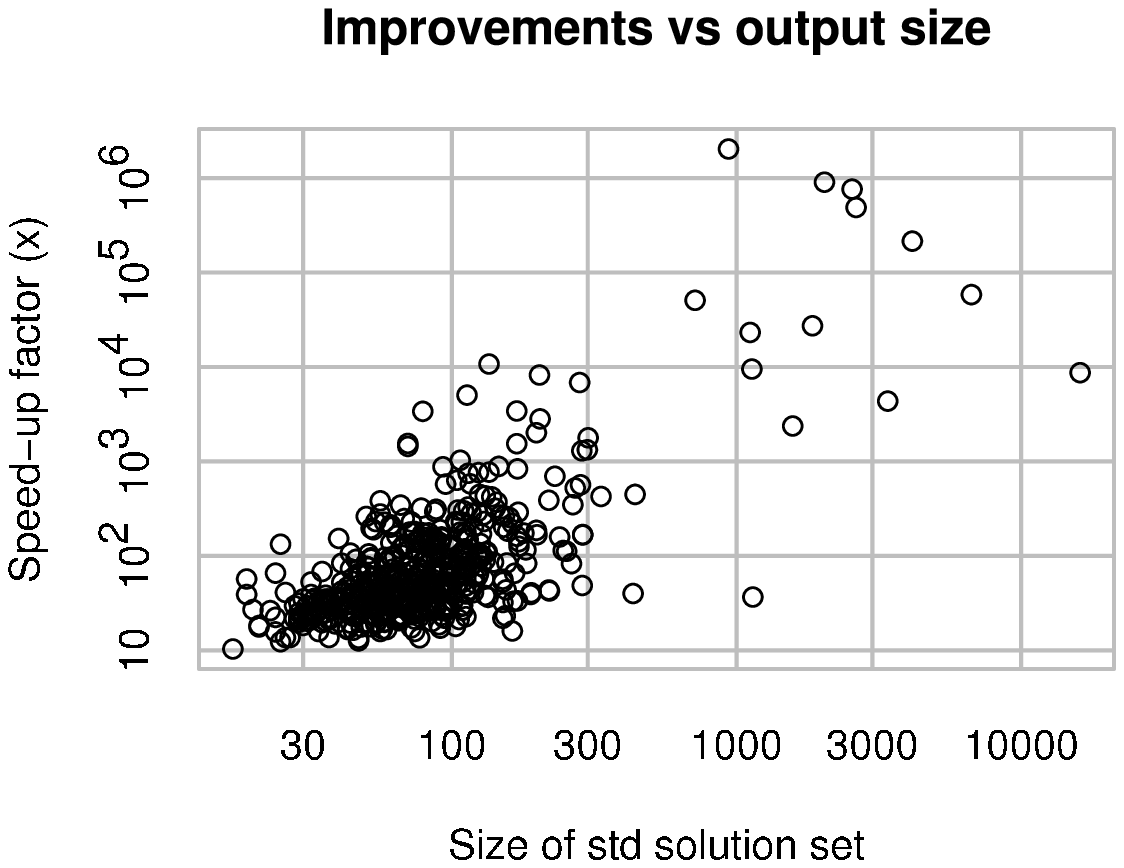}
    }
    \caption{Speed improvement factors for the new Algorithm~\ref{a-enumstd}}
    \label{fig-largefactor}
    \end{figure}

    One striking observation is how small the solution sets are, given
    that the triangulations range from $n=9$ to $n=25$ tetrahedra and
    that the sizes of the solution sets can grow exponentially in $n$.
    We examine this effect in greater detail in \cite{burton09-extreme}.

    If we focus on cases with unusually large input and output sets---those
    points that escape the dense clouds at the left of each plot---we
    find again that the improvements are particularly strong.  Amongst the
    triangulations with input size $>100$ the improvement factors range
    from over $4\,000$ to over $2\,000\,000$.  Likewise, with the
    exception of just one triangulation, those with output size $>500$
    have improvements ranging from over $2\,000$ to over $2\,000\,000$.
    The lone exception has output size $1\,141$ and an improvement
    factor of $37$.
\end{itemize}

Our final tests examine the feasibility of Conjecture~\ref{cj-complexity}.
Recall that this conjecture states that the running time for the
quadrilateral-to-standard solution set
conversion algorithm (Algorithm~\ref{a-quadtostd})
is at worst polynomial in the size of the output.  For this to occur we
must avoid the combinatorial explosion in the sizes of the intermediate
lists $L_0,L_1,\ldots,L_m$.

Figure~\ref{fig-explosion} measures the extent of this combinatorial
explosion.  Specifically, for each triangulation we measure the size of the
{\em maximal} list divided by the size of the {\em final} list---if
we have a combinatorial explosion we expect this ratio to be
very large, and if not then we expect it to remain close to one.
We then bin these measurements into small ranges and plot the resulting
frequencies in a histogram (so in each of the three plots, the sum of
the heights of the bars is always 500).
We take these measurements not only for Algorithm~\ref{a-quadtostd}
but also for the old direct enumeration algorithm
in both quadrilateral and standard coordinates.

\begin{figure}[htb]
\centerline{
    \includegraphics[scale=0.55]{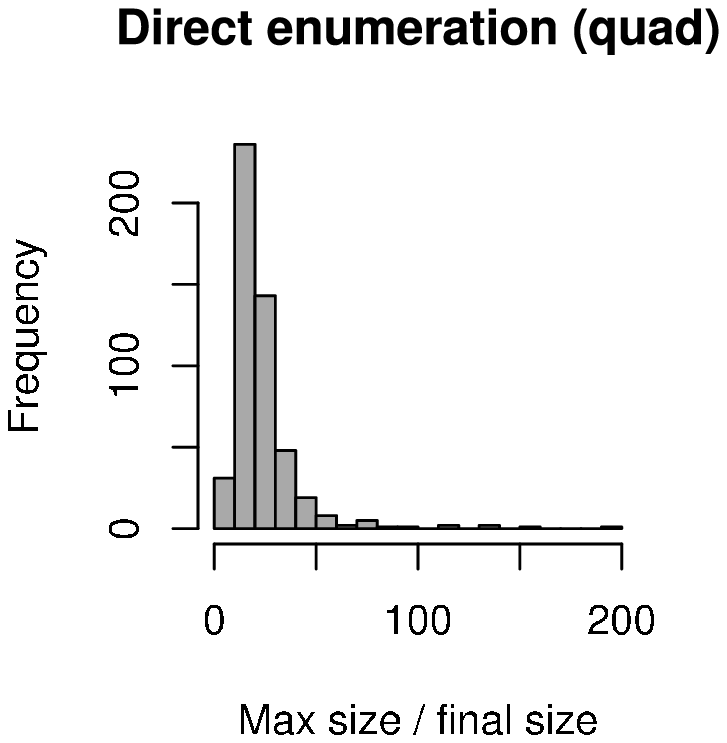}\quad
    \includegraphics[scale=0.55]{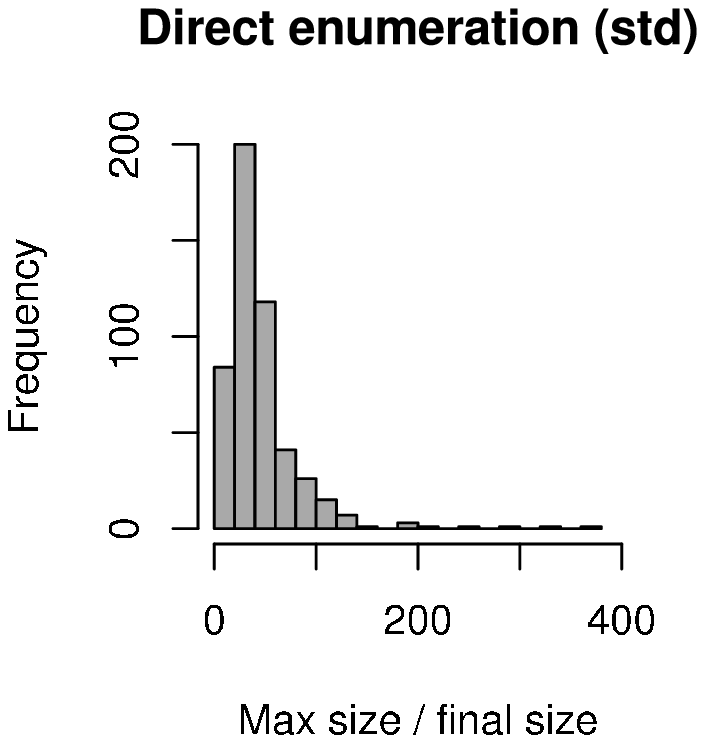}\quad
    \includegraphics[scale=0.55]{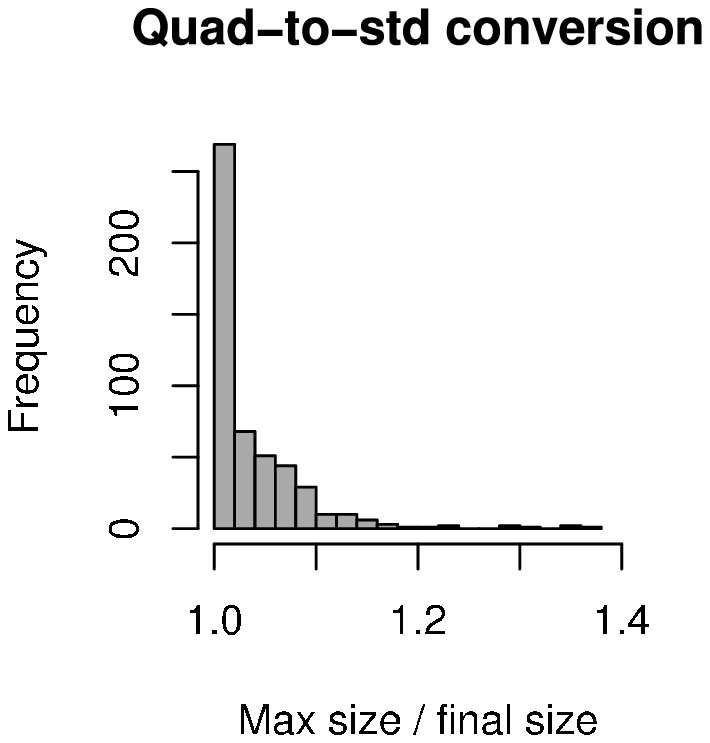}
}
\caption{The combinatorial explosion for enumeration and conversion
    algorithms}
\label{fig-explosion}
\end{figure}

What we see is exactly what we hope for.  With the old direct enumeration
algorithms, the maximal list can grow to hundreds of times the output
size (and perhaps larger, recalling that for the eight worst cases
the direct enumeration in standard coordinates was prematurely
terminated after 30~days).  For Algorithm~\ref{a-quadtostd} this ratio
is never greater than $\frac32$.  That is, the behaviour we see is
consistent with the intermediate lists being bounded by a {\em linear}
function of the output size.

Figure~\ref{fig-complexity} tests our conjecture
more directly by plotting the running time of
Algorithm~\ref{a-quadtostd} against the output size $k'$ (the size of
the standard solution set).  Once again, both axes use a log scale so
that the data points are more evenly distributed.

\begin{figure}[htb]
\centerline{\includegraphics[scale=0.55]{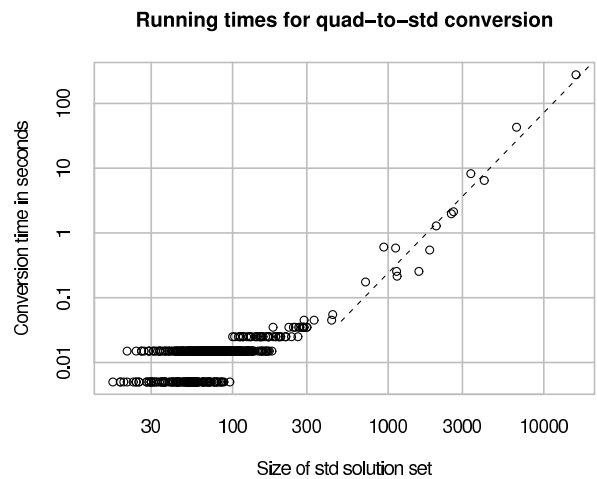}}
\caption{The running time for Algorithm~\ref{a-quadtostd}
    as a function of output size}
\label{fig-complexity}
\end{figure}

It is reasonable to ignore all points where the running time is under
$0.1$ seconds, since the clock resolution is only $0.01$ seconds (once
again we see horizontal bands of points where the running times are the
smallest that can be measured).  Not only does the clock resolution
introduce large relative errors for these points, but they are also highly
susceptible to what would otherwise be negligible tasks, such as
initialising data structures at the beginning of the algorithm, or
extracting algebraic information from the triangulation.

Focusing our attention therefore on the points with time $>0.1\,\mathrm{s}$
(or equivalently, with output size $>500$), we find that the points
follow what appears to be a straight line.  If $t$ is the running time,
this corresponds to an equation of the form
$\log t = \alpha \log k' + \beta$, or
equivalently, $t \propto k'^\alpha$.  That is, the time does indeed
appear polynomial in the output size $k'$.

We can measure the degree of this polynomial by performing a linear
regression.  This regression is indicated by the dashed line in
Figure~\ref{fig-complexity}; its equation is approximately
\[ \log t = 2.4729 \times \log k' - 18.5016 .\]
That is, the running time appears to be a little under
$t \propto k'^{2.5}$.  The adjusted correlation coefficient for
this regression is $r \simeq 0.96$, indicating an extremely good linear fit.

Note that $t \propto k'^{2.5}$ is quite reasonable, given
the structure of Algorithm~\ref{a-quadtostd}.  If we assume that
each list $L_i$ has size $O(k')$, then each inductive step $L_i \to L_{i+1}$
involves at least $O(k'^2)$ iterations through the innermost loop
(running through all $\mathbf{u} \in S_+$ and $\mathbf{w} \in S_-$).
This inner loop can in turn take $O(k')$ time as it tests for adjacency
by searching for an appropriate $\mathbf{z} \in L_r$; however,
Fukuda and Prodon \cite{fukuda96-doubledesc} note that such searches
often terminate early, and our additional test on the quadrilateral
constraints means that many such searches can be avoided entirely.
We therefore expect an average running time of between $O(k'^2)$
and $O(k'^3)$, which is precisely what we see.

One might observe that we have neglected the number of tetrahedra $n$
entirely in this empirical discussion of Conjecture~\ref{cj-complexity}.
Of course $n$ features implicitly in the size of the output, since
each vector in the standard solution set has dimension $7n$.
We focus on $k'$ here because it spans several orders of magnitude,
ranging from 17 to 16\,106; in contrast, $n$ merely ranges from
9 to 25.  Since the size of the standard solution set can grow
exponentially in $n$ (and this is also found to be true in the average
case \cite{burton09-extreme}), it is reasonable to expect $k'$ to
become the dominating factor in the running time.

%% file: notation.tex
\appendix

\section*{Appendix: Notation}

Throughout this paper we introduce a number of symbols that are used in the
statements and proofs of results.  For convenience, the following tables
list the key symbols and where they are defined.

\subsubsection*{Sets and Vector Spaces:}

\centerline{\begin{tabular}{l|l|l}
    \em Symbol & \em Meaning & \em Point of definition \\
    \hline
    $O^d$ & Non-negative orthant & Definition~\ref{d-soln-space} \\
    $J^d$ & Projective hyperplane & \\
    $\stdproj$ & Standard projective solution space & \\
    $\quadproj$ & Quadrilateral projective solution space & \\
    \hline
    $\surfaces$ & All embedded normal surfaces & Notation~\ref{notn-spaces} \\
    $\surfacescan$ & All canonical embedded normal surfaces \\
    $\vspacea$, $\qspacea$ & Admissible vectors in
        $\R^{7n}$ or $\R^{3n}$ \\
    $\vspaceza$, $\qspaceza$ & Admissible integer vectors in
        $\Z^{7n}$ or $\Z^{3n}$ & \\
    $\vspaceac$, $\vspacezac$ & Admissible canonical vectors in
        $\R^{7n}$ or $\Z^{7n}$ & \\
    \hline
    $\ainv{r}$, $\binv{C}$ & Used for loop invariants in
        Algorithm~\ref{a-quadtostd} &
        Equations~(\ref{eqn-a-defn}) and~(\ref{eqn-b-defn})
\end{tabular}}

\subsubsection*{Maps:}

\centerline{\begin{tabular}{l|l|l}
    \em Symbol & \em Meaning & \em Point of definition \\
    \hline
    $\ell(\cdot)$ & Vertex link (surface or vector) & Definition~\ref{d-link} \\
    $\vrep{\cdot}$, $\qrep{\cdot}$ &
        Vector representation & Definition~\ref{d-vecrep} \\
    $\proj{\,\cdot\,}$, $\vproj{\cdot}$, $\qproj{\cdot}$ &
        Projective image & Definition~\ref{d-projimage} \\
    $\sigma_v(\cdot)$, $\sigma_q(\cdot)$ &
        Represented surface & Definition~\ref{d-represented} \\
    $\kappa_s(\cdot)$, $\kappa_v(\cdot)$ &
        Canonical part (surface or vector) &
        Definition~\ref{d-canonical-part} \\
    $\qmap{\cdot}$ & Quadrilateral projection & Definition~\ref{d-proj-ext} \\
    $\vmap{\cdot}$ & Canonical extension & Definition~\ref{d-proj-ext} \\
    $\kappa_v^{(i)}(\cdot)$ & Partial canonical part &
        Definition~\ref{d-partial-canonical} \\
    $\tau_i(\cdot)$ & Truncation & Definition~\ref{d-truncation} \\
    $\adm{\cdot}$ & Semi-admissible part & Definition~\ref{d-semi-adm}
\end{tabular}}